\documentclass[11 pt]{amsart}
\usepackage{hyperref}
\usepackage{tikz}
\usepackage[margin=3 cm]{geometry}
\usepackage[all,cmtip]{xy}
\usepackage[psamsfonts]{amssymb}
\usepackage{amscd,amsfonts,amssymb,amsmath,nccmath}
\newtheorem{theorem}{Theorem}[section]

\newtheorem{lemma}[theorem]{Lemma}
\newtheorem{prop}[theorem]{Proposition}
\theoremstyle{definition}

\newtheorem{example}[theorem]{Example}
\newtheorem{remark}[theorem]{Remark}
\numberwithin{equation}{subsection}
\theoremstyle{plain}
\newtheorem*{ack}{Acknowledgement}
\newtheorem{conjecture}{Conjecture}
\newtheorem*{question}{Questions}

\newtheorem{corollary}[theorem]{Corollary}

\newcommand{\im}{\operatorname{Im}}
\newcommand{\Ker}{\operatorname{Ker}}
\newcommand{\Hom}{\operatorname{Hom}}
\newcommand{\Ha}{\operatorname{H}}
\newcommand{\Has}{\operatorname{HS}}
\newcommand{\res}{\operatorname{res}}
\newcommand{\sres}{\operatorname{s-res}}
\newcommand{\lres}{\operatorname{\lambda-res}}
\newcommand{\cores}{\operatorname{cores}}
\newcommand{\scores}{\operatorname{s-cores}}
\newcommand{\lcores}{\operatorname{\lambda-cores}}
\newcommand{\Ca}{\operatorname{C}}
\newcommand{\Ta}{\operatorname{T}}
\newcommand{\Za}{\operatorname{Z}}
\newcommand{\Cas}{\operatorname{CS}}
\newcommand{\KS}{\operatorname{KS}}
\newcommand{\tr}{\operatorname{tr}}
\newcommand{\Tr}{\operatorname{Tr}}
\newcommand{\Ka}{\operatorname{K}}
\newcommand{\Ba}{\operatorname{B}}
\newcommand{\Bas}{\operatorname{BS}}
\newcommand{\id}{\operatorname{id}}

\def\Z{\mathbb Z}

\numberwithin{equation}{subsection}

\begin{document}
\title{Exterior and symmetric (co)homology of groups}
\author[V. G. Bardakov]{Valeriy G. Bardakov}
\author[M. V. Neshchadim]{Mikhail V. Neshchadim}
\author[M. Singh]{Mahender Singh}

\date{\today}
\address{Sobolev Institute of Mathematics and Novosibirsk State University, Novosibirsk 630090, Russia.}
\address{Novosibirsk State Agrarian University, Dobrolyubova street, 160, Novosibirsk, 630039, Russia.}
\address{Regional Scientific and Educational Mathematical Center of Tomsk State University, 36 Lenin Ave., Tomsk, Russia.}
\email{bardakov@math.nsc.ru}

\address{Sobolev Institute of Mathematics and Novosibirsk State University, Novosibirsk 630090, Russia.}
\address{Regional Scientific and Educational Mathematical Center of Tomsk State University, 36 Lenin Ave., Tomsk, Russia.}
\email{neshch@math.nsc.ru}

\address{Department of Mathematical Sciences, Indian Institute of Science Education and Research (IISER) Mohali, Sector 81,  S. A. S. Nagar, P. O. Manauli, Punjab 140306, India.}
\email{mahender@iisermohali.ac.in}


\begin{abstract}
The paper investigates exterior and symmetric (co)homologies of groups. We introduce symmetric homology of groups and compute exterior and symmetric (co)homologies of some finite groups. We also compare the classical, exterior and symmetric (co)homologies. Finally, we derive restriction and corestriction homomorphisms for exterior cohomology.
\end{abstract}
\maketitle

\section{Introduction}
The classical (co)homology theory of groups has its origins both in algebra and topology. We refer to Brown \cite{Brown2010} for an excellent historical introduction of the subject. Several (co)homology theories of groups have been proposed ever since the subject was formalised by Eilenberg and MacLane \cite{Eilenberg1, Eilenberg2}. The purpose of this paper is to investigate two of these (co)homology theories, namely, the exterior and the symmetric (co)homologies of groups.
\par

Motivated by the construction of a homology theory for crossed simplicial groups by Fiedorowicz and Loday \cite{Fiedorowicz}, Staic \cite{Staic1} introduced the notion of a $\Delta$-group $\Gamma(X)$ for a topological space $X$. Given a group $G$ and a $G$-module $A$, Staic defined an action of the symmetric group $\Sigma_{*+1}$ on the cochain group $\Ca^*(G, A)$ used for the classical group cohomology and proved it to be compatible with the corresponding coboundary maps. The subcomplex of invariant elements $\{ \Ca^*(G, A)^{\Sigma_{*+1}} \}$ gives a new cohomology theory, denoted $\Has^*(G, A)$, and called the symmetric cohomology. It is proved in \cite{Staic1} that the $\Delta$-group $\Gamma(X)$ is determined by the action of $\pi_1(X)$ on $\pi_2(X)$, and an element of $\Has^3\big(\pi_1(X), \pi_2(X)\big)$. The inclusion of cochain complexes $\Ca^*(G, A)^{\Sigma_{*+1}} \hookrightarrow \Ca^*(G, A)$ gives a natural map $$\alpha^*: \Has^*(G, A) \longrightarrow \Ha^*(G, A).$$
\par

A continuous analogue of symmetric cohomology of topological groups, and a smooth analogue for Lie groups was proposed by Singh \cite{Singh}. Among other things, it was proved that the symmetric continuous cohomology of a profinite group with coefficients in a discrete module can be computed as the direct limit of the symmetric cohomology groups of its finite quotients with appropriate coefficients. Some related questions were discussed in \cite{BGSVW}. Developing Staic's work further, Todea  \cite{Todea} gave explicit constructions for the transfer, restriction and conjugation maps, and showed under some mild hypotheses that the family $\{\Has^*(H,A)\}_{H\le G}$ has the structure of a Mackey functor.
\par

Not so well-known is the interesting work \cite{Zarelua} of Zarelua, wherein he introduced exterior cohomology $\Ha^*_{\lambda}(G, A)$ and exterior homology $\Ha_*^{\lambda}(G, A)$ of groups. The exterior cohomology groups also come equipped with a natural map $$\beta^*: \Ha^*_{\lambda}(G, A) \longrightarrow \Ha^*(G, A),$$
and have the property that if $G$ is a finite group of order $d$, then $\Ha^i_{\lambda}(G, A)=0$ for all $i \ge d$. In \cite{Pirashvili}, Pirashvili has explored in detail the connections between the maps $\alpha^*$ and $\beta^*$. Among other results, she constructed a natural map $$\gamma^*:\Ha^*_{\lambda}(G, A) \longrightarrow \Has^*(G, A)$$
which is a split monomorphism such that the diagram
$$
\xymatrix{
\Ha^*_{\lambda}(G, M)  \ar[d]^{\beta^*} \ar[r]^{\gamma^*} & \Has^*(G, A) \ar[dl]^{\alpha^*}\\
\Ha^*(G, A) &}
$$
commutes. Further, in a recent work \cite{Pirashvili1}, the third symmetric cohomology has been linked  to crossed modules with certain properties.
\par

The purpose of this paper is to investigate exterior and symmetric (co)homologies of groups which are difficult to compute due to lack of computational machinery. A symmetric homology of groups is introduced and the same is computed for some small order groups. A major part of the paper deals with computations of exterior (co)homology of some finite groups. We also derive restriction and corestriction maps for exterior and symmetric cohomologies.

\par

The paper is organised as follows. Section \ref{sec2} recalls the construction of classical (co)homology which also serves our purpose of setting up the notation. In Section \ref{sec3}, we recall the definition of symmetric cohomology and propose a symmetric homology of groups. In Section \ref{sec4}, we recall the construction of exterior (co)homology of groups. In Section \ref{sec5}, we compare the classical, exterior and symmetric (co)homologies of groups leading to new (co)homologies and make basic observations about these (co)homologies. In Section \ref{sec6}, we compute symmetric homology of some small order groups. Section \ref{sec7} contains computations of exterior (co)homology of some finite groups.  Finally, in Section \ref{sec8}, we derive restriction and corestriction maps for exterior and symmetric cohomologies.
\par

Throughout the paper, we use $\partial_n$ for boundary maps and $\delta^n$ for coboundary maps when their definitions and underlying complexes are clear from the context. Further, following standard terminology, modules over the integral group ring $\mathbb{Z}[G]$ are referred as $G$-modules.
\medskip


\section{Classical (co)homology}\label{sec2}
We begin by recalling some standard definitions and facts from \cite{Brown}.

\subsection{Standard resolution for classical (co)homology}\label{sec-standard-resol}
Let  $G$ be a group and $\mathbb{Z}[G]$ the group ring of $G$ with integer coefficients. Denote  by $G^{n+1}$ the $(n+1)$-fold cartesian product of $G$, that is,
$$
G^{n+1}=\big\{\, (g_0,g_1,\ldots,g_n) \, | \, g_0,g_1,\ldots,g_n\in G \,\big\}.
$$
For each $n \geq 0$, let $\Ba_n(G):=\mathbb{Z}[G^{n+1}]$ be the free abelian group with basis  $G^{n+1}$. Then $G$ acts on $\Ba_n(G)$ by the rule
$$h(g_0,g_1,\ldots,g_n) = (hg_0,hg_1,\ldots,hg_n),$$
where $h, g_0, g_1, \ldots, g_n\in G$. This action turns $\Ba_n(G)$ into a left   $G$-module. Notice that $\Ba_n(G)$ is a free left $G$-module with basis
$$
\big\{(1,g_1,\ldots,g_n)\mid   g_1, \ldots, g_n\in G \big\}.
$$
For each $n \geq 1$, the map $\partial_n : \Ba_n(G) \longrightarrow \Ba_{n-1}(G)$ of left $G$-modules defined by
\begin{equation}\label{resolution-differential}
\partial_n(g_0,g_1,\ldots,g_n) =
 \sum\limits_{i=0}^{n}(-1)^i \big(g_0,\ldots,\widehat{g_i},\ldots,g_n\big),
\end{equation}
where $(g_0,\ldots,\widehat{g_i},\ldots,g_n)=(g_0,\ldots,g_{i-1},g_{i+1},\ldots,g_n)$ satisfy $\partial_n \partial_{n+1}=0$. Note that $\Ba_0(G)=\mathbb{Z}[G]$ and the augmentation homomorphism  $\varepsilon: \Ba_0(G) \longrightarrow \mathbb{Z}$ is given by
$$
\varepsilon \big( \sum n_i g_i \big)=\sum n_i,
$$
where  $n_i\in \mathbb{Z}$ and $g_i\in G$.
\par

The left $G$-modules $\Ba_n(G)$, $n \geq 0$, together with the boundary maps  $\partial_n$, $n \geq 1$, and  the augmentation homomorphism $\varepsilon $ forms the standard free resolution
\begin{equation}\label{standard-resolution}
\cdots \stackrel{\partial_{n+2}}{\longrightarrow} \Ba_{n+1}(G)
    \stackrel{\partial_{n+1}}{\longrightarrow} \Ba_n(G) \stackrel{\partial_{n}}{\longrightarrow} \cdots
       \longrightarrow \Ba_1(G) \stackrel{\partial_{1}}{\longrightarrow} \Ba_0(G) \stackrel{\varepsilon}{\longrightarrow} \mathbb{Z}
\end{equation}
of the trivial $G$-module $\mathbb{Z}$.
\medskip

If $A$ is a right  $G$-module, then the \textit{classical homology groups} $\Ha_n(G,A)$, $n \geq 0$, are defined as homology groups of the chain  complex
\begin{equation}\label{eqno2}
\cdots \stackrel{\partial_{n+2}}{\longrightarrow} A \underrel{\otimes}{G} \Ba_{n+1}(G)
    \stackrel{\partial_{n+1}}{\longrightarrow} A \underrel{\otimes}{G} \Ba_n(G) \stackrel{\partial_{n}}{\longrightarrow} \cdots
               \stackrel{\partial_{2}}{\longrightarrow} A \underrel{\otimes}{G} \Ba_1(G)
                     \stackrel{\partial_{1}}{\longrightarrow} A \underrel{\otimes}{G} \Ba_0(G) \stackrel{\partial_{0}}{\longrightarrow} 0,
\end{equation}
where the boundary map $\partial_n$ is induced by the boundary map of \eqref{resolution-differential}.
\par

If $A$ is a left  $G$-module, then the \textit{classical cohomology groups} $\Ha^n(G,A)$, $n \geq 0$, are defined as the cohomology groups of the cochain  complex
\begin{equation}\label{eqno3}
0 \longrightarrow \Hom_G \big(\Ba_0(G), A\big)\stackrel{\delta^{0}}{\longrightarrow} \Hom_G \big(\Ba_1(G), A\big)
\stackrel{\delta^{1}}{\longrightarrow} \cdots    \stackrel{\delta^{n-1}}{\longrightarrow} \Hom_G \big(\Ba_n(G), A\big)
            \stackrel{\delta^{n}}{\longrightarrow}  \cdots,
\end{equation}
where the coboundary map  $\delta^n$ is induced by the boundary map  of \eqref{resolution-differential}.
\medskip

\subsection{Another complex for classical cohomology}\label{another-complex}
The classical group cohomology can also be obtained using another cochain complex which we describe next. Set $\Ca^{-1}(G, A)=0$, and $\Ca^0(G,A)=A$ viewed as maps from the trivial group to the group $A$. For each integer $n \geq 1$, let
\begin{equation}\label{function-cochain-ord}
\Ca^n(G, A)=\big\{\sigma: G^n \to A \big\}.
\end{equation}
Then the coboundary map (same notation being used) $$\delta^n:\Ca^n(G,A) \longrightarrow \Ca^{n+1}(G,A)$$ given by
\begin{eqnarray}\label{cocycle-eqn}
& &\delta^n(\sigma)(g_1,\,\dots\,,\,g_{n+1})\\ 
\nonumber &=& g_1\sigma(g_2,\,\dots,\,g_{n+1}) +\sum_{k=1}^n (-1)^k \sigma(g_1,\,\dots,\,g_k g_{k+1},\,\dots,\,g_{n+1})
+(-1)^{n+1}\sigma(g_1,\,\dots,\,g_n)
\end{eqnarray}
for $\sigma \in \Ca^n(G,A)$ and $(g_1,\,\dots,\,g_{n+1}) \in G^{n+1}$, turns $\{\Ca^*(G,A), \delta^*\}$ into a cochain complex. Observe that, for each $n \ge 0$,
\begin{eqnarray*}
\Hom_G\big(\Ba_n(G), A\big) &=& \Hom_G \big(\mathbb{Z}[G^{n+1}], A\big)\\
&=& \big\{f: G^{n+1} \to A~|~ f(gg_0, gg_1, \ldots, g g_n)=gf(g_0, g_1, \ldots, g_n) \big\}.
\end{eqnarray*}
Then the map
\begin{equation}\label{iso-two-complexes}
\psi^n:\Hom_G\big(\Ba_n(G), A\big) \longrightarrow \Ca^n(G,A)
\end{equation}
defined by
$$\psi^n(f)(g_1,\ldots , g_n)=f(1,g_1,g_1g_2, \ldots ,g_1g_2\ldots g_n)$$
induces an isomorphism of cochain complexes $$\psi^*:\Hom_G\big(\Ba_*(G), A\big) \longrightarrow \Ca^*(G,A).$$ Thus, the cochain complex $\{\Ca^*(G,A), \delta^*\}$ also gives the classical group cohomology defined earlier. This fact will be used in Section \ref{sec8} where we define (co)restriction homomorphisms for exterior and symmetric cohomologies.
\medskip

\section{Symmetric (co)homology}\label{sec3}

\subsection{Symmetric cohomology}\label{sym-cohom}

We recall the definition of symmetric cohomology originally introduced by Staic \cite{Staic1}. Let $G$ be a group and $A$ a $G$-module. Let $\Ca^n(G, A) = \{ \sigma : G^n \to A \}$ and $\delta^n : \Ca^n(G, A) \longrightarrow \Ca^{n+1}(G, A)$ be as in \eqref{cocycle-eqn}. Define $d^j : \Ca^n(G, A) \longrightarrow \Ca^{n+1}(G, A)$ by
\begin{eqnarray*}
d^0 (\sigma) (g_1, \ldots, g_{n+1}) &=& g_1 \sigma(g_2, \ldots, g_{n+1}),\\
d^j (\sigma) (g_1, \ldots, g_{n+1}) &=&  \sigma(g_1, \ldots, g_j g_{j+1}, \ldots, g_{n+1})~ \textrm{for}~1 \leq j \leq n,\\
d^{n+1} (\sigma) (g_1, \ldots, g_{n+1}) &=& \sigma(g_1, \ldots, g_{n+1}).
\end{eqnarray*}
Then we notice that
$$
\delta^n = \sum_{j=0}^{n+1} (-1)^j d^j.
$$
Staic constructed an action of the symmetric group $\Sigma_{n+1}$ on $\Ca^n(G, A)$ which is compatible with the coboundary maps $\delta^n$. If $\tau_{i}$ denote the transposition $(i,i+1)$ for $1 \leq i \leq n$ and $\sigma \in \Ca^n(G, A)$, then
\begin{eqnarray*}
(\tau_1 \sigma) (g_1, g_2, g_3, \ldots, g_{n})  &=& -g_1 \sigma(g_1^{-1}, g_1 g_2, g_3, \ldots, g_{n+1}),\\
(\tau_i \sigma) (g_1, g_2, g_3, \ldots, g_{n})  &=&  -\sigma(g_1, \ldots, g_{i-2}, g_{i-1} g_i, g_i^{-1}, \ldots, g_{n})~\textrm{for} ~1 < i < n,\\
(\tau_n \sigma) (g_1, g_2, g_3, \ldots, g_{n})  &=& -\sigma(g_1, g_2, g_3, \ldots, g_{n-1} g_{n} g_n^{-1}).
\end{eqnarray*}

The cohomology  $\Has^n(G,A)$ of the subcomplex of  invariants $\Cas^n (G, A):= \Ca^n (G, A)^{\Sigma_{n+1}}$ is called the \textit{symmetric cohomology} of $G$ with coefficients in $A$. By \cite[Lemma 3.1]{Staic2}, the map
$$
\Has^2(G, A) \to \Ha^2(G, A)
$$
induced by the inclusion $\Cas^n (G, A) \hookrightarrow \Ca^n (G, A)$ of cochain complexes is injective.
\medskip

\begin{example}
A very few examples of computations of symmetric cohomology are known. By \cite[Remark 5.4]{Staic1}, $\Has^{2k}(\mathbb{Z}_2, \mathbb{Z})=0$,  $\Has^{2}(\mathbb{Z}_2, \mathbb{Z}_n)=0$ and  $\Has^{2}(\mathbb{Z}_4, \mathbb{Z})=\mathbb{Z}_2$. Further, it is known due to \cite[Lemma 3.10]{Pirashvili} that
$$
\Has^{k}(\mathbb{Z}_2, \mathbb{Z}_2)=
\left\{
\begin{array}{ll}
\mathbb{Z}_2 & \mbox{if}\,\, k=0~\textrm{or}~k \equiv 1 \mod 4,\\
0 & \mbox{otherwise}.
\end{array}
\right.
$$
\end{example}


\subsection{Alternate definition of symmetric cohomology}
We recall an alternate approach to symmetric cohomology by Pirashvili \cite{Pirashvili} which shows that the symmetric cohomology of Staic can be defined in a more natural way using the standard resolution.

Let $G$ be a group and 
$$
\bold{T}_n(G) := \mathbb{Z}[G]^{\otimes (n+1)}
$$
the $G$-module generated by the set
$$
\big\{ g_0 \otimes g_1 \otimes \cdots \otimes g_n~|~g_i \in G \big\}.
$$
Then the symmetric group $\Sigma_{n+1}$ acts on $\bold{T}_n(G)$ with the action defined on the generators by
$$
\tau_j (g_0 \otimes g_1 \otimes \cdots \otimes g_j \otimes g_{j+1}\otimes \cdots \otimes g_{n}) = g_0 \otimes g_1 \otimes \cdots \otimes g_{j+1} \otimes g_{j}\otimes \cdots \otimes g_{n},
$$
where $\tau_j = (j, j+1)$ for $0 \le j \le n-1$. We also have  homomorphisms $$d_i : \bold{T}_n(G) \longrightarrow \bold{T}_{n-1}(G)$$ defined on generators by
$$
d_i (g_0 \otimes g_1 \otimes \cdots \otimes g_i \otimes \cdots \otimes g_{n}) = g_0 \otimes g_1 \otimes \cdots \otimes \widehat{g_i} \otimes \cdots \otimes g_{n}.
$$
It is not difficult to see that setting $$\partial_n := \sum_{i=0}^n (-1)^i d_i$$ gives a chain complex $\big\{\bold{T}_*(G), \partial_* \big\}$. If $A$ is a left  $G$-module, then applying the functor $\Hom_G(-, A)$  gives the cochain complex $\{\Ka^* (G, A), \delta^* \}$, where $\Ka^* (G, A):= \Hom_G \big(\bold{T}_*(G), A\big)$ and $\delta^*$ is the induced coboundary map. The action of  $\Sigma_{n+1}$ on $\bold{T}_n(G)$ induces an action on $ \Ka^n (G, A)$ given by $$\tau_i f(g_0 \otimes g_1\otimes \cdots \otimes g_i \otimes g_{i+1} \otimes \cdots \otimes g_n)= f(g_0 \otimes g_1\otimes \cdots \otimes g_{i+1} \otimes g_i \otimes \cdots \otimes g_n).$$
\par

A function $\sigma \in \Ka^n(G, A)$ is called  {\it skew-symmetric} if
$$
\tau_i \sigma= - \sigma
$$
for each $0 \le i \le   n-1$. Let $\KS^n(G,A)$ be the $G$-submodule of $\Ka^{n}(G,A)$ consisting of skew-symmetric functions.  
Since the coboundary map $\delta^n$ keeps $\KS^n(G,A)$ invariant, we obtain a cochain complex  $\{\KS^*(G, A), \delta^*\}$. By \cite[Lemma 3.5]{Pirashvili}, the cochain complex $\{\KS^*(G, A), \delta^*\}$ is isomorphic to Staic's cochain complex defining symmetric cohomology.

\begin{lemma} \label{staic-alternate}
The $n$-th cohomology group  of the cochain complex   $\{\KS^*(G,A), \delta^* \}$ is isomorphic to the $n$-th symmetric cohomology $\Has^n(G, A)$.
\end{lemma}
\medskip


\subsection{Symmetric homology}\label{sec-symm-homo}
We conclude this section by introducing symmetric homology of groups. For a group  $G$, we have $\mathbb{Z}[G^n] \cong \mathbb{Z}[G]^{\otimes n}$ as $G$-modules. Recall the standard free resolution 
$$
\cdots \stackrel{\partial_{n+2}}{\longrightarrow} \Ba_{n+1}(G)
    \stackrel{\partial_{n+1}}{\longrightarrow} \Ba_{n}(G) \stackrel{\partial_{n}}{\longrightarrow} \cdots
       \longrightarrow \Ba_{1}(G) \stackrel{\partial_{1}}{\longrightarrow} \Ba_{0}(G) \stackrel{\varepsilon}{\longrightarrow} \mathbb{Z},
$$
of the trivial $G$-module $\mathbb{Z}$, where $\Ba_n(G)=\mathbb{Z}[G]^{\otimes (n+1)}$ for  $n \ge 0$. Consider the $G$-submodule  $\Bas_n(G)$  of $\Ba_n(G)$ that is generated by all alternative sums of the form
$$
\sum\limits_{\sigma\in \Sigma_{n+1}} \mathrm{sign}(\sigma) ~\big(g_{\sigma(0)}\otimes \cdots \otimes g_{\sigma(n)} \big),
$$
where  $g_0,\ldots, g_n \in G$. By \cite[Lemma 3.2]{Zarelua}, the left $G$-modules  $\Bas_n(G)$, $n \ge 0$, with the induced boundary maps forms a chain complex
\begin{equation}\label{symmetric-homology-resolution}
\cdots \stackrel{\partial_{n+2}}{\longrightarrow} \Bas_{n+1}(G)
    \stackrel{\partial_{n+1}}{\longrightarrow} \Bas_{n}(G) \stackrel{\partial_{n}}{\longrightarrow} \cdots
       \longrightarrow \Bas_{1}(G) \stackrel{\partial_{1}}{\longrightarrow} \Bas_{0}(G) \stackrel{\varepsilon}{\longrightarrow} \mathbb{Z}.
\end{equation}
\par

If $A$ is a right   $G$-module, then the \textit{symmetric homology groups} $\Has_n(G,A)$, for $n \geq 0$, are defined as homology groups of the chain complex
\begin{equation}\label{symmetric-homology-chain-complex}
\ldots \stackrel{\partial_{n+2}}{\longrightarrow} A \underrel{\otimes}{G} \Bas_{n+1}(G)
    \stackrel{\partial_{n+1}}{\longrightarrow} A \underrel{\otimes}{G} \Bas_{n}(G) \stackrel{\partial_{n}}{\longrightarrow} \cdots
               \stackrel{\partial_{2}}{\longrightarrow} A \underrel{\otimes}{G} \Bas_{1}(G)
                     \stackrel{\partial_{1}}{\longrightarrow} A \underrel{\otimes}{G} \Bas_{0}(G) \longrightarrow 0,
\end{equation}
where  $\partial_n$ is the induced boundary map. 
\medskip

\section{Exterior (co)homology}\label{sec4}

\subsection{Exterior (co)homology} We recall the definition of exterior (co)homology introduced by Zarelua \cite{Zarelua}. If $V$ is a left $R$-module, where $R$ is a ring (not necessary commutative), then the exterior algebra  $\Lambda^*(V)$ is defined as a quotient of the tensor algebra  $\Ta^*(V)$ by the ideal generated by the elements
$$
v_1\otimes \cdots \otimes v_i \otimes \cdots \otimes v_j \otimes \cdots \otimes v_n,
$$
where  $v_k\in V$, $v_i = v_j$ for $i\not= j$ and $n\geq 2$.  The ring $R$ acts on  $\Ta^*(V)$ diagonally as
$$
x (v_1\otimes \cdots  \otimes v_n)=xv_1\otimes \cdots  \otimes xv_n,
$$
where $x \in R$ and  $v_1, \dots ,v_n \in V$. This turns $\Lambda^*(V)$ into a left $R$-module. If $V$ is a free  $R$-module with a basis $\{e_i\}_{i\in I}$,  where $I$ is linearly ordered,  then  $\Lambda^{n}(V)$
is a free  $R$-module with the basis
$$
\big\{e_{i_1}\wedge \cdots  \wedge e_{i_n} ~\mid ~ i_1 < \cdots < i_n \big\}.
$$
It is evident that if  $I$ is finite, then  $\Lambda^{n}(V)=0$ for  $n > |I|$.
\par

The standard projective resolution \eqref{standard-resolution} of the trivial $G$-module $\mathbb{Z}$ can be rewritten as
$$\cdots \longrightarrow  \bold{T}_n(G)  \stackrel{\partial_n}{\longrightarrow} \bold{T}_{n-1}(G)  \stackrel{\partial_{n-1}}{\longrightarrow}  \cdots  \stackrel{\partial_2}{\longrightarrow} \bold{T}_1(G) \stackrel{\partial_1}{\longrightarrow} \Z[G]  \stackrel{\varepsilon}{\longrightarrow} \mathbb{Z}.$$
We set $\bold\Lambda_n(G):= \Lambda^{n+1}\big(\mathbb{Z}[G]\big)$. If we replace the tensor algebra by the exterior algebra, then we again obtain a resolution which, in general, is not projective. It is easy to prove the following result \cite[Lemma 3.1]{Zarelua}.

\begin{lemma}
The boundary map $\partial_n:\bold{T}_n(G) \longrightarrow \bold{T}_{n-1}(G)$ induces a boundary map (with the same notation) $$\partial_n: \bold\Lambda_n(G) \longrightarrow \bold\Lambda_{n-1}(G)$$ given by
$$
\partial_n(g_0\wedge g_1\wedge \cdots \wedge g_n) =
 \sum\limits_{i=0}^{n}(-1)^i \big(g_0\wedge \cdots\wedge \widehat{g_i}\wedge \cdots \wedge g_n \big).
$$
\end{lemma}

Consider the resolution
\begin{equation}\label{exterior-resolution}
\cdots \stackrel{\partial_{n+1}}{\longrightarrow} \bold\Lambda_n(G)
   \stackrel{\partial_{n}}{\longrightarrow} \bold\Lambda_{n-1}(G)
     \stackrel{\partial_{n-1}}{\longrightarrow}   \cdots \stackrel{\partial_{1}}{\longrightarrow}
       \bold\Lambda_0(G) \stackrel{\varepsilon}{\longrightarrow} \mathbb{Z}\longrightarrow 0.
\end{equation}
Tensoring the complex by a  right $G$-module $A$ gives the chain complex

\begin{equation}
\cdots \stackrel{\partial_{n+1}}{\longrightarrow} A \underrel{\otimes}{G} \bold\Lambda_n(G)
\stackrel{\partial_{n}}{\longrightarrow} A\underrel{\otimes}{G} \bold\Lambda_{n-1}(G)
  \stackrel{\partial_{n-1}}{\longrightarrow}   \cdots
  \stackrel{\partial_{1}}{\longrightarrow} A\underrel{\otimes}{G}\bold\Lambda_0(G)
  \stackrel{\partial_{0}}{\longrightarrow}  0.
\end{equation}
The homology groups  $\Ha_*^{\lambda}(G,A)$ of the preceding chain complex are called the \textit{exterior homology groups} of $G$ with coefficients in $A$. If $A$ is a left $G$-module, then applying the  functor
$\Hom_G(-, A)$ on the resolution \eqref{exterior-resolution} yields the cochain complex
\begin{equation}
0\longrightarrow \Hom_G\big( \bold\Lambda_0(G), A\big)
\stackrel{\delta^{0}}{\longrightarrow}  \cdots
  \stackrel{\delta^{n-2}}{\longrightarrow}     \Hom_G\big( \bold\Lambda_{n-1}(G), A\big)
  \stackrel{\delta^{n-1}}{\longrightarrow}  \Hom_G\big( \bold\Lambda_n(G), A\big)  \stackrel{\delta^{n}}{\longrightarrow}  \cdots,
\end{equation}
whose  cohomology groups  $\Ha_{\lambda}^*(G,A)$ are called the \textit{exterior cohomology groups} of $G$ with coefficients in $A$.

\begin{example}
Not many examples of computations of exterior (co)homology are known due to lack of sufficient theory. By \cite[p. 414]{Pirashvili}, if $p$ is a prime, then
$$
\Ha_{\lambda}^k(\mathbb{Z}_p,A)=
\left\{
\begin{array}{ll}
\Ha^k(\mathbb{Z}_p,A)  & \mbox{if}\,\, k \le p-1, \\
0 & \mbox{if}\,\, k \ge p. \\
\end{array}
\right.
$$
\end{example}

\medskip

\subsection{Map from symmetric to exterior homology}
Let $G$ be a group, $\Ba_n(G)= \mathbb{Z}[G]^{\otimes(n+1)}$ and $\bold\Lambda_n(G)= \Lambda^{n+1}(\mathbb{Z}[G])$ as $G$-modules. For $g_0, g_1, \ldots, g_n \in G$, set
$$\mu_n(g_0 \otimes g_1 \otimes \cdots \otimes g_n) := \sum_{\sigma \in \Sigma_{n+1}} \rm{sign} \big(\sigma) \big(g_{\sigma(0)}\otimes g_{\sigma(1)} \otimes  \cdots \otimes g_{\sigma(n)} \big).$$
Let $\Bas_n(G)$ be the $G$-submodule of $\Ba_n(G)$ as in Subsection \ref{sec-symm-homo}. The standard boundary maps \eqref{resolution-differential} induce boundary maps  on the subcomplexes  $\Bas_*(G)$ and $\bold\Lambda_*(G)$. By \cite[Lemma 3.3]{Pirashvili}, $\{\Ba_*(G), \partial_*\}$ and $\{\Lambda_*(G), \partial_*\}$ are resolutions of the trivial $G$-module $\mathbb{Z}$. On the other hand, $\{\Bas_*(G), \partial_*\}$ is not a resolution.
\par

Let $\lambda_n: \Ba_n(G) \longrightarrow \bold\Lambda_n(G)$ be the natural projection given by $$\lambda_n (g_0 \otimes g_1 \otimes \cdots \otimes g_n)= g_0 \wedge g_1 \wedge \cdots \wedge g_n.$$ The universal property of exterior product yields a map $\nu_n: \bold\Lambda_n(G) \longrightarrow \Bas_n(G)$ given by $$\nu_n(g_0 \wedge g_1 \wedge \cdots \wedge g_n)= \mu_n(g_0 \otimes g_1 \otimes \cdots \otimes g_n).$$ This gives the commutative diagram
$$
\xymatrix{
\Ba_n(G)  \ar[r]^{\lambda_n} \ar[d]^{\mu_n} & \bold\Lambda_n(G) \ar[dl]^{\nu_n}\\
 \Bas_n(G). &}
$$

The following identities are easy to check and will be used in computations of symmetric homology in Section \ref{sec6}.

\begin{lemma}\label{identities}
 The following holds for each $n \ge 1$:
\begin{enumerate}
\item $\lambda_n \nu_n=(n+1)!$.
\item $\lambda_{n-1} \partial_n = \partial_n \lambda_n$.
\item $\partial_n \mu_n= (n+1) \mu_{n-1} \partial_n$.
\item $\partial_n \nu_n= (n+1) \nu_{n-1} \partial_n$.
\end{enumerate}
\end{lemma}
\par

\begin{prop}
There exists a group homomorphism $\Has_*(G, A) \longrightarrow \Ha_*^\lambda(G,A).$
\end{prop}

\begin{proof}
The surjective chain map $1 \otimes \lambda_*: A \otimes_G \Ba_*(G)  \longrightarrow A \otimes_G \bold\Lambda_*(G)$ gives a homomorphism of homology groups
$$\lambda_*: \Ha_*(G,A) \longrightarrow \Ha_*^\lambda(G,A).$$

Similarly, the inclusion $\Bas_*(G) \hookrightarrow \Ba_*(G)$ gives a chain map $\iota_*: A \otimes_G \Bas_*(G) \hookrightarrow A\otimes_G B_*(G)$, which further gives a homomorphism of homology groups
$$\iota_*: \Has_*(G, A) \longrightarrow \Ha_*(G,A).$$
The composite $\lambda_* \iota_*: \Has_*(G, A) \longrightarrow \Ha_*^\lambda(G,A)$ is the desired homomorphism.
\end{proof}
\medskip

\section{New cohomologies}\label{sec5}
We compare the classical, exterior and symmetric cohomologies using natural maps of their defining cochain complexes. This gives new cohomologies for groups and we make some basic observations about these cohomologies.
\par

Let $G$ be a group and $A$ a $G$-module. In view of the isomorphism \eqref{iso-two-complexes}, the cochain complex $\{\Ka^*(G, A), \delta^* \}$, where
\begin{eqnarray*}
\Ka^n(G, A) &= & \Hom_{G}\big(\mathbb{Z}[G^{n+1}], A\big)\\
 &= & \big\{ \sigma: \mathbb{Z}[G^{n+1}] \to A~|~ \sigma\big(g(g_0,\ldots,g_n)\big)=g\sigma(g_0,\ldots,g_n)~\mathrm{for~ all}~ g, g_i \in G \big\}
\end{eqnarray*}
and $\delta^n$ the standard coboundary map of \eqref{eqno3}, gives the classical cohomology $\Ha^*(G,A)$.
\par
Let us define
\begin{eqnarray*}
\KS^n (G,A) &:=& \big\{ \sigma \in \Ka^n(G, A)~|~ \sigma(g_0,\ldots,g_i,g_{i+1},\ldots,g_n)=-\sigma(g_0,\ldots,g_{i+1},g_i,\ldots,g_n)\\
& &\mathrm{for~ all}~0\leq i< n~\mathrm{and}~g_0, g_1, \ldots, g_n \in G \big\}
\end{eqnarray*}
and
\begin{eqnarray*}
\Ka^n_\lambda (G,A) &:=& \big\{\sigma \in \KS^n(G,A)~|~ \sigma(g_0,\ldots,g_i,g_i,\ldots,g_n)=0\\
& &\mathrm{for~ all}~0\leq i< n~\mathrm{and}~g_0, g_1, \ldots, g_n \in G \big\}.
\end{eqnarray*}

By Lemma \ref{staic-alternate}, the cohomology of the cochain complex $\{\KS^* (G,A), \delta^*\}$ is the symmetric cohomology $\Has^*(G,A)$. Similarly, by \cite[Lemma 3.5]{Pirashvili}, the cohomology of the cochain complex $\{\Ka^*_\lambda (G,A), \delta^* \}$ is the exterior cohomology $\Ha_\lambda^*(G,A)$.
\par

If the groups and the modules are clear from the context, for brevity, we write the complexes as $\Ka^*, \KS^*, \Ka^*_\lambda$, and their cohomologies as $\Ha^*, \Has^*, \Ha^*_\lambda$, respectively.
\medskip

\subsubsection{The cohomology $\Ha^*_{s \lambda}$}
We denote the cohomology groups of the quotient cochain complex $\{ \KS^*/\Ka^*_\lambda, ~\overline{\delta}^*\}$ by $\Ha^*_{s \lambda}$, where $\overline{\delta}^*$ is the induced coboundary map. We note that the cohomology $\Ha^*_{s \lambda}$ was originally introduced in \cite[Section 3.2]{Pirashvili} where it is denoted as $\Ha^*_\delta$. The following result follows from \cite[Theorem 3.9 and Proposition 3.6]{Pirashvili}.

\begin{prop}
Let $G$ be a group and $A$ a $G$-module. Then the following hold:
\begin{enumerate}
\item There exists an isomorphism
$$ \Has^n(G,A) \cong \Ha^n_\lambda(G,A) \oplus \Ha^n_{s \lambda}(G,A)$$
for each $n \ge 0$.
\item  $\Ha^n_{s \lambda}(G,A)=0$ for all $0 \le n \le 4$. 
\item If $A$ has no element of order 2, then  $\Ha^n_{s \lambda}(G,A)=0$ for all $n \ge 0$.
\end{enumerate}
\end{prop}
\medskip

\subsubsection{The cohomology $\Ha^*_{c \lambda}$}
The quotient cochain complex $\{ \Ka^*/\Ka^*_\lambda, \overline{\delta}^*\}$ gives cohomology groups, which we denote by $\Ha^*_{c \lambda}$. The short exact sequence of cochain complexes
$$ 0 \longrightarrow \Ka^*_\lambda \longrightarrow \Ka^* \longrightarrow \Ka^*/\Ka^*_\lambda \longrightarrow 0,$$
gives the long exact sequence of cohomology groups
\begin{equation}\label{long-exact-ext-ord}
0  \to \Ha^0_\lambda \to \Ha^0 \to \Ha^0_{c \lambda} \to \Ha^1_\lambda \to \Ha^1 \to \Ha^1_{c \lambda} \to  \Ha^2_\lambda \to \Ha^2 \to \Ha^2_{c \lambda} \to \cdots.
\end{equation}

\begin{prop}
Let $G$ be a group and $A$ a $G$-module. Then the following hold:
\begin{enumerate}
\item $\Ha^0_{c \lambda}(G, A)=0=\Ha^1_{c \lambda}(G, A)$.
\item If $G$ has no element of finite order, then $\Ha^n_{c \lambda}(G, A)=0$ for all $n \ge 0$.
\end{enumerate}
\end{prop}

\begin{proof}
By \cite{Zarelua}, $\Ha^0_\lambda(G,A) = \Ha^0(G,A)$ and $\Ha^1_\lambda(G,A) = \Ha^1(G,A)$. By \cite[Theorem 3.9]{Pirashvili}, the homomorphism $\Ha^2_\lambda(G,A) \to \Has^2(G,A)$ is an isomorphism. But, $\Has^2(G,A) \to \Ha^2(G,A)$ is an embedding. Hence the homomorphism $\Ha^2_\lambda(G,A) \to \Ha^2(G,A)$ is an embedding being the composite $\Ha^2_\lambda(G,A) \to \Has^2(G,A) \to \Ha^2(G,A)$. The assertion now follows from the long exact sequence \eqref{long-exact-ext-ord}.
\par
By \cite[Corollary 4.4(iii)]{Pirashvili},  if $G$ has no element of finite order, then the homomorphism $\Ha^n_\lambda(G,A) \to \Ha^n(G,A)$ is an isomorphism for all $n \ge 0$, and the result again follows from \eqref{long-exact-ext-ord}.
\end{proof}

\subsubsection{The cohomology $\Ha^*_{c s}$}
As in the preceding cases, let us denote the cohomology of the quotient complex $\{\Ka^*/\KS^*, \overline{\delta}^*\}$ by $\Ha^*_{c s}$. The short exact sequence of cochain complexes
$$ 0 \longrightarrow \KS^* \longrightarrow \Ka^* \longrightarrow \Ka^*/\KS^*  \longrightarrow 0$$
gives the long exact sequence
\begin{equation}\label{long-exact-symm-ord}
0  \to \Has^0 \to \Ha^0 \to \Ha^0_{c s} \to \Has^1 \to \Ha^1 \to \Ha^1_{c s} \to  \Has^2 \to \Ha^2 \to \Ha^2_{c s} \to  \cdots.
\end{equation}

\begin{prop}
Let $G$ be a group and $A$ a $G$-module.  Then the following hold:
\begin{enumerate}
\item $\Ha^0_{c s}(G, A)=0=\Ha^1_{c s}(G, A)$.
\item If $n + 1$ is not a zero divisor and the equation $n!\,x = a$ has exactly one solution in $A$, then there exists a short exact sequence of  groups $$0 \longrightarrow \Has^n (G,A) \longrightarrow \Ha^n(G,A) \longrightarrow \Ha^n_{c s}(G,A)  \longrightarrow 0$$ for each $n \ge 0$.
\end{enumerate}
\end{prop}

\begin{proof}
By definition of symmetric cohomology, $\Has^0(G,A) = \Ha^0(G,A)$. By \cite[Proposition 2.1]{Todea}, $\Has^1(G,A) = \Ha^1(G,A)$. Further, by \cite[Lemma 3.1]{Staic2}, the homomorphism $\Has^2(G,A) \to \Ha^2(G,A)$ is injective, and the result now follows from the long exact sequence \eqref{long-exact-symm-ord}.
\par

By \cite[Proposition 4.1]{Staic2}, for such a group $A$, the homomorphism $\Has^n(G,A) \to \Ha^n(G,A)$ is injective for each $n \ge 0$, and the result follows from \eqref{long-exact-symm-ord}.
\end{proof}
\medskip

\section{Computations of symmetric homology}\label{sec6}

\subsection{Some general results}
We begin with some basic but general results.

\begin{prop}\label{symm-homo-o-dim}
Let $G$ be a group and  $A=\mathbb{Z}[G]$ viewed as a right $G$-module. Then the following holds:
\begin{enumerate}
\item $\Has_0(G, A)= \mathbb{Z}[G]/2 \Delta(G)$, where $\Delta(G)$ is the augmentation ideal of $\mathbb{Z}[G]$.
\item If $G$ is of order $n$, then $\Has_i(G, A)= 0$ for all $i \ge n-1$.
\end{enumerate}
\end{prop}
\begin{proof}
Notice that $\bold\Lambda_0(G)=\Bas_0(G)= \mathbb{Z}[G]$. In view of Lemma \ref{identities}(4), we have the following commutative diagram

$$
\xymatrix{
\bold\Lambda_1(G)  \ar[r]^{\partial_1}  \ar[d]^{\nu_1} & \bold\Lambda_0(G) \ar[d]^{2\id}  \ar[r]^{\varepsilon} & \mathbb{Z}  \\
 \Bas_1(G) \ar[r]^{\partial_1}  & \Bas_0(G) \ar[r]^{\varepsilon} & \mathbb{Z}.}
$$

The top row of the diagram being part of \eqref{exterior-resolution} is exact. It follows that
$$
\im\big(\partial_1:\Bas_1(G) \to \Bas_0(G) \big) = 2 \Ker \big(\varepsilon: \Bas_0(G) \to \mathbb{Z} \big)= 2 \Delta(G).$$
Hence $\Has_0(G, A)= \mathbb{Z}[G]/2 \Delta(G)$ which proves (1).
\par
Let  $G=\{g_1, g_2, \ldots, g_n\}$. Since $\Bas_i(G)=0 $ for all $i \ge n$, it follows that $\Has_i(G, A)= 0$ for all $i \ge n$. Further, we have $$\Bas_{n-1}(G)= \textrm{mod}_{\mathbb{Z}[G]} \big\langle \mu_{n-1}(g_1\otimes g_2\otimes \cdots \otimes g_n) \big\rangle\cong \mathbb{Z},$$ since
\begin{eqnarray*}
g \,\mu_{n-1}(g_1\otimes g_2 \otimes \cdots \otimes g_n) &=& \mu_{n-1}\big(g_{\sigma(1)}\otimes g_{\sigma(2)}\otimes \cdots \otimes g_{\sigma(n)}\big)~\textrm{for some}~\sigma\in \Sigma_n\\
&=&\mu_{n-1}\big(g_1\otimes g_2 \otimes \cdots \otimes g_n\big).
\end{eqnarray*}
By Lemma \ref{identities}(3), we have
\begin{eqnarray*}
\partial_{n-1}\big(\mu_{n-1}(g_1\otimes g_2\otimes \cdots \otimes g_n)\big) &= & n \mu_{n-2} \big( \partial_{n-1}(g_1\otimes g_2\otimes \cdots \otimes g_n)\big)\\
&= & n \mu_{n-2} \big( \sum_{j=1}^n (-1)^j(g_1\otimes  \cdots \otimes\hat{g_j} \otimes \cdots \otimes g_n)\big)\\
&= & n  \sum_{j=1}^n (-1)^j \mu_{n-2}(g_1\otimes  \cdots \otimes \hat{g_j} \otimes \cdots \otimes g_n)\\
& \neq & 0,
\end{eqnarray*}
since the summands are independent. Thus, $\Ker(\partial_{n-1} )=0$, and hence $\Has_{n-1}(G, A)=0$, completing the proof of assertion (2).
\end{proof}

\begin{prop}\label{symm-homo-o-dim-trivial}
If $G$ is a group and $A$ a trivial right $G$-module, then $\Has_0(G, A)= A$. Further, if $G$ is of order $n$, then $\Has_i(G, A)= 0$ for all $i \ge n$.
\end{prop}

\begin{proof}
Recall that $\Bas_1(G)$ is generated by $\big\{\mu_1(g \otimes h)~|~g, h \in G \big\}$. For $g, h \in G$ and $a \in A$, we have
\begin{eqnarray*}
\partial_1\big(a \otimes \mu_1(g \otimes h)\big) & =& a \otimes \partial_1\big(\mu_1(g \otimes h)\big)\\
& =& a \otimes 2 \partial_1(g \otimes h)~\textrm{by Lemma \ref{identities}(3)}\\
& =& a \otimes 2 (h -g)\\
& =& 0.
\end{eqnarray*}
Thus, $\im(\partial_1)=0$, and hence $\Has_0(G, A)= A \otimes_G \Bas_0(G)= A$. The second assertion is obvious from the definition.
\end{proof}
\medskip

\subsection{Groups of order 2 and 3}
If $G=\langle g \mid g^2=1 \rangle$, then the complex \eqref{symmetric-homology-resolution} takes the form
\begin{equation*}
0        \longrightarrow \Bas_{1}(G) \stackrel{\partial_{1}}{\longrightarrow} \Bas_{0}(G) \stackrel{\partial_{0}}{\longrightarrow} 0.
\end{equation*}
Taking  $A=\mathbb{Z}[G]$ as a right $G$-module, by Proposition \ref{symm-homo-o-dim}, we obtain
$$
\Has_i(G, A)=
\left\{
\begin{array}{ll}
\mathbb{Z}[G]/2 \Delta(G)\cong \mathbb{Z} \oplus \mathbb{Z}_2  & \mbox{if}\,\, i=0, \\
0 & \mbox{if}\,\, i \ge 1. \\
\end{array}
\right.
$$
\medskip

If  $A$ is a trivial right $G$-module, then the chain complex \eqref{symmetric-homology-chain-complex} becomes
\begin{equation*}
0        \longrightarrow A \otimes_G \Bas_{1}(G) \stackrel{\partial_{1}}{\longrightarrow} A \otimes_G \Bas_{0}(G) \stackrel{\partial_{0}}{\longrightarrow} 0.
\end{equation*}
Notice that, for $a \in A$, we have
\begin{eqnarray*}
a \otimes \mu_1(1 \otimes g) & =& a \otimes (1 \otimes g -g \otimes 1)\\
& =& a \otimes (g(g \otimes 1) -g \otimes 1)\\
& =& 0,
\end{eqnarray*}
and hence $A \otimes_G \Bas_{1}(G)=0$. Thus, by Proposition \ref{symm-homo-o-dim-trivial}, we obtain
$$
\Has_i(G, A)=
\left\{
\begin{array}{ll}
A  & \mbox{if}\,\, i=0, \\
0 & \mbox{if}\,\, i \ge 1. \\
\end{array}
\right.
$$
\medskip

Next we consider the cyclic group $G=\langle g \mid g^3=1 \rangle$ of order 3, for which the chain complex is
\begin{equation*}
  0        \longrightarrow \Bas_{2}(G)      \stackrel{\partial_{2}}{\longrightarrow} \Bas_{1}(G) \stackrel{\partial_{1}}{\longrightarrow} \Bas_{0}(G) \stackrel{\partial_{0}}{\longrightarrow} 0.
\end{equation*}
Take  $A=\mathbb{Z}[G]$ as a right $G$-module. By Lemma \ref{identities}(4), $\im(\partial_2)=3 \Ker(\partial_1)$. A direct computation yields
$$\Ker(\partial_1)= \textrm{mod}_{\mathbb{Z}[G]} \big\langle \mu_1(1 \otimes g)+ \mu_1(g \otimes g^2) +\mu_1(g^2 \otimes 1)\big\rangle,$$ and $$g \big( \mu_1(1 \otimes g)+ \mu_1(g \otimes g^2) +\mu_1(g^2 \otimes 1)\big)= \mu_1(1 \otimes g)+ \mu_1(g \otimes g^2) +\mu_1(g^2 \otimes 1).$$
Thus, $\Ker(\partial_1) \cong \mathbb{Z}$, and hence $\Has_1(G, A)\cong \mathbb{Z}_3$. This together with Proposition \ref{symm-homo-o-dim}
gives
$$
\Has_i(G, A)=
\left\{
\begin{array}{ll}
\mathbb{Z}[G]/2 \Delta(G)\cong \mathbb{Z} \oplus \mathbb{Z}_2  & \mbox{if}\,\, i=0, \\
\mathbb{Z}_3 & \mbox{if}\,\, i = 1, \\
0 & \mbox{if}\,\, i \ge 2.
\end{array}
\right.
$$
\medskip

Finally, we consider an arbitrary trivial $G$-module  $A$. Then we have the chain complex
\begin{equation*}
0        \longrightarrow A \otimes_G \Bas_{2}(G)      \stackrel{\partial_{2}}{\longrightarrow}    A \otimes_G \Bas_{1}(G) \stackrel{\partial_{1}}{\longrightarrow} A \otimes_G \Bas_{0}(G) \stackrel{\partial_{0}}{\longrightarrow} 0,
\end{equation*}
where $ A \otimes_G \Bas_{i}(G) \cong A$ for $i=0,1,2$. For $a \in A$, we have
$$\partial_1 \big(a \otimes \mu_1(1 \otimes g)\big) =a \otimes 2(g -1)= 0,$$ which shows that $\im(\partial_1)=0$. Similarly, we obtain
\begin{eqnarray*}
\partial_2\big( a \otimes \mu_2(1 \otimes g \otimes g^2) \big)& =& a \otimes 3 \mu_1\partial_2(1 \otimes g \otimes g^2)\\
& =& 3a \otimes \big(\mu_1(1 \otimes g)+ \mu_1(g \otimes g^2) +\mu_1(g^2 \otimes 1)\big)\\
& =& 9a \otimes \mu_1(1 \otimes g).
\end{eqnarray*}
Thus, $a \otimes \mu_2(1 \otimes g \otimes g^2) \in \Ker(\partial_2)$ if and only if $9a=0$. Hence, the homology groups of $G$ are as follows
$$
\Has_i(G, A)=
\left\{
\begin{array}{ll}
A  & \mbox{if}\,\, i=0, \\
A/9A  & \mbox{if}\,\, i=1, \\
\textrm{Tor}_9(A)  & \mbox{if}\,\, i=2, \\
0 & \mbox{if}\,\, i \ge 3. \\
\end{array}
\right.
$$
\medskip


\section{Computations of exterior (co)homology}\label{sec7}

In this section, we compute exterior homology of some finite groups.

\subsection{Arbitrary finite group}
Let $G=\left\{ g_1,g_2,\ldots , g_n  \right\}$ be a finite group of order $n$, where $g_1=e$ is the identity element.
Then we have
$$
\bold\Lambda_{n-1}(G)= \mathrm{mod}_{\mathbb{Z}[G]}\left\langle
     \,  g_1\wedge g_2\wedge \cdots \wedge g_{n} \,
     \right\rangle.
$$

\begin{lemma} \label{l6}
$\bold\Lambda_{n-2}(G)= \mathrm{mod}_{\mathbb{Z}[G]}\left\langle
     \,  g_2\wedge g_3\wedge \cdots \wedge g_{n} \,  \right\rangle.$
\end{lemma}

\begin{proof}
Let $G_1=\left\{ g_2, \ldots , g_n  \right\}$. We claim that $g_i G_1 \neq g_j G_1$  for  $i \neq j$. If  $g_i G_1 = g_j G_1$, then $g_j^{-1}g_i G_1 =  G_1$. Since $g_i \neq g_j$, then $g_j^{-1}g_i \neq  e$, and hence $g_i^{-1}g_j \in G_1$. The equality $G_1=g_j^{-1}g_i G_1$ implies that $e \in G_1$, which is a contradiction. Using the sets  $gG_1$, $g\in G$, one can write uniquely up to a sign all $n-1$ forms $g_1\wedge \cdots \wedge \widehat{g_i}\wedge \cdots \wedge g_{n}$ in $\bold\Lambda_{n-2}(G)$. The preceding argument shows that all these forms can be obtained from $g_2\wedge g_3\wedge \cdots \wedge g_{n}$ by multiplication by some element $g\in G$.
\end{proof}

Let us set
$$
\alpha:=g_1\wedge g_2\wedge \cdots \wedge g_{n}~\textrm{and}~ \beta:=g_2\wedge g_3\wedge \cdots \wedge g_{n}.
$$
Next we derive a formula for the boundary map
$$
\partial_{n-1}:\bold\Lambda_{n-1}(G) \longrightarrow \bold\Lambda_{n-2}(G).
$$
Let  $\kappa : G \longrightarrow \Sigma_n$ be the Cayley representation of $G$ given by
$\kappa (g)=\sigma$, where  $\sigma \in \Sigma_n$ and
$$\sigma (g_1,\ldots,g_n)=\big(g_{\sigma(1)},\ldots,g_{\sigma(n)}\big).
$$
Let $\pi : \Sigma_n  \longrightarrow \mathbb{Z}_2$ be the natural projection $$\pi(\sigma)= \mathrm{sign}(\sigma).$$
 A group $G$ is called {\it oriented} if the composition   $\pi \circ \kappa : G \longrightarrow \mathbb{Z}_2$
is the trivial homomorphism. If the composition  $\pi \circ \kappa : G  \longrightarrow \mathbb{Z}_2$
is a non-trivial homomorphism, then  $G$ is called {\it non-oriented}.
\par

Next we show that the definition does not depend on the linear order on the elements of  $G$. Let $N = \sum\limits_{i=1}^{n}g_i$ be the norm element in the integral group ring $\mathbb{Z}[G]$ of $G$. Then

\begin{theorem} \label{t2}
The following formula holds:
$$
\partial_{n-1}(\alpha)=
\left\{
\begin{array}{ll}
 N \beta  & \mbox{if}\,\, G \,\, \mbox{is oriented}, \\
  \left(\sum\limits_{i=1}^{n}\mathrm{sign}\big(\kappa(g_i)\big)g_i\right)\beta &
   \mbox{if}\,\, G \,\, \mbox{is non-oriented}. \\
\end{array}
\right.
$$
\end{theorem}

\begin{proof}
If $G$ is an oriented group, then for each  $i=1,\ldots,n$, we have
\begin{eqnarray*}
\alpha & = & g_1\wedge g_2\wedge \cdots \wedge g_{n}\\
&=&   g_i (g_1\wedge g_2\wedge \cdots \wedge g_{n})\\
&=& (-1)^{i-1} g_i ( g_2\wedge\cdots\wedge g_i \wedge g_1 \wedge g_{i+1} \wedge  \cdots \wedge g_{n}).
\end{eqnarray*}
Hence, we get
\begin{eqnarray*}
g_1\wedge \cdots \wedge \widehat{g_i}\wedge \cdots \wedge g_{n} & = & (-1)^{i-1} g_i ( g_2\wedge\cdots\wedge g_i \wedge  g_{i+1} \wedge  \cdots \wedge g_{n})\\
 & = & (-1)^{i-1} g_i\beta.
\end{eqnarray*}
From this, we obtain
\begin{eqnarray*}
\partial_{n-1}(\alpha) & = &   \sum\limits_{i=1}^{n}(-1)^{i-1} g_1\wedge \cdots \wedge \widehat{g_i}\wedge \cdots \wedge g_{n}\\
& = &        \sum\limits_{i=1}^{n}(-1)^{2(i-1)}g_i \beta \\
& = & N \beta.
\end{eqnarray*}
\par

If $G$ is a non-oriented group, then for each  $i=1,\ldots,n$ the following equality holds
\begin{eqnarray*}
\alpha &=& g_1\wedge g_2\wedge \cdots \wedge g_{n}\\
&=&   \mathrm{sign}\big(\kappa(g_i)\big)g_i (g_1\wedge g_2\wedge \cdots \wedge g_{n})\\
&=& (-1)^{i-1} \mathrm{sign}\big(\kappa(g_i)\big)g_i     ( g_2\wedge\cdots\wedge g_i \wedge g_1 \wedge g_{i+1} \wedge  \cdots \wedge g_{n}).
\end{eqnarray*}
Thus, we obtain
\begin{eqnarray*}
g_1\wedge \cdots \wedge \widehat{g_i}\wedge \cdots \wedge g_{n} &=& (-1)^{i-1} \mathrm{sign}\big(\kappa(g_i)\big)
     g_i ( g_2\wedge\cdots\wedge g_i \wedge  g_{i+1} \wedge  \cdots \wedge g_{n})\\
&=& (-1)^{i-1} \mathrm{sign}\big(\kappa(g_i)\big) g_i\beta.
\end{eqnarray*}
This gives
\begin{eqnarray*}
\partial_{n-1}(\alpha) &=&   \sum\limits_{i=1}^{n}(-1)^{i-1} g_1\wedge \cdots \wedge \widehat{g_i}\wedge \cdots \wedge g_{n}\\
&=&   \sum\limits_{i=1}^{n}(-1)^{2(i-1)}\mathrm{sign}\big(\kappa(g_i)\big)g_i \beta\\
&=&      \left(\sum\limits_{i=1}^{n}\mathrm{sign}\big(\kappa(g_i)\big)g_i\right)\beta,
\end{eqnarray*}
which is desired.
\end{proof}

\begin{corollary}
The following holds:
\begin{enumerate}
\item  A group being oriented (non-oriented) does not depend on the labelling of its elements.
\item If the group  $G$ is non-oriented, then  it has even order.
\end{enumerate}
\end{corollary}

\begin{proof}
(1) Since changing the labelling of the elements of the group $G$ only changes the signs of the forms  $\alpha$, $\beta$ and the sum $\sum\limits_{i=1}^{n}g_i$ is an invariant of  $G$, the result follows.
\par
(2) If $G$ is non-oriented, then there is an epimorphism of $G$ onto the cyclic group of order 2, and hence the order of $G$ is even.
\end{proof}

To understand $\partial_{n-1}:\bold\Lambda_{n-1}(G) \longrightarrow \bold\Lambda_{n-2}(G)$ we determine its kernel and image. Up to a sign we can assume that
$$
\partial_{n-1}(\alpha)= \left(\sum\limits_{i=1}^{n}(-1)^{i-1}g_i \right)\beta.
$$

\begin{theorem} \label{t3}
Let  $A$ be a right  $G$-module.
\begin{enumerate}
\item If $G$ is oriented, then
$$
\im \big(\partial_{n-1} : A\underrel{\otimes}{G}\bold\Lambda_{n-1}(G) \rightarrow A\underrel{\otimes}{G}\bold\Lambda_{n-2}(G)\big)=
\left\{
\, a \left(\sum\limits_{i=1}^{n}g_i\right)\otimes \beta\, \mid a \in A\,
\right\},
$$
\begin{eqnarray*}
\Ha_{n-1}^{\lambda}(G,A)&=& \Ker \big(\partial_{n-1} : A\underrel{\otimes}{G}\bold\Lambda_{n-1}(G)
\rightarrow A\underrel{\otimes}{G}\bold\Lambda_{n-2}(G)\big)\\
&=& \left\{ \, a \otimes \alpha\, \mid a \in A , \,\,a \left(\sum\limits_{i=1}^{n}g_i\right)\otimes \beta=0 \,\right\}.
\end{eqnarray*}
\item[]
\item If $G$ is non-oriented, then
$$
\im \big(\partial_{n-1} : A\underrel{\otimes}{G}\bold\Lambda_{n-1}(G) \rightarrow A\underrel{\otimes}{G}\bold\Lambda_{n-2}(G)\big)=
\left\{
\, a \left(\sum\limits_{i=1}^{n}(-1)^{i-1}g_i\right)\otimes \beta\, \mid a \in A\,
\right\},
$$
\begin{eqnarray*}
\Ha_{n-1}^{\lambda}(G,A) &=&  \Ker \big(\partial_{n-1} : A\underrel{\otimes}{G}\bold\Lambda_{n-1}(G)
\rightarrow A\underrel{\otimes}{G}\bold\Lambda_{n-2}(G)\big)\\
&=& \left\{ \, a \otimes \alpha\, \mid a \in A , \,\,a \left(\sum\limits_{i=1}^{n}(-1)^{i-1}g_i\right)\otimes \beta=0 \,\right\}.
\end{eqnarray*}
\end{enumerate}
\end{theorem}
\bigskip

\begin{corollary}
Let $A$ be a trivial right $G$-module.
\begin{enumerate}
\item If  $G$ is oriented, then
$$
\im \big(\partial_{n-1} : A\underrel{\otimes}{G}\bold\Lambda_{n-1}(G) \rightarrow A\underrel{\otimes}{G} \bold\Lambda_{n-2}(G)\big)=
nA \otimes \beta,
$$

\begin{eqnarray*}
\Ha_{n-1}^{\lambda}(G,A) &=& \Ker \big(\partial_{n-1} : A\underrel{\otimes}{G}\bold\Lambda_{n-1}(G)
\rightarrow A\underrel{\otimes}{G}\bold\Lambda_{n-2}(G) \big)\\
&=&\big\{\, a \otimes \alpha\, \mid a \in A , \,\,na \otimes \beta=0 \,\big\},
\end{eqnarray*}
\item[]
\item If  $G$ is non-oriented, then
$$
\im \big(\partial_{n-1} : A\underrel{\otimes}{G}\bold\Lambda_{n-1}(G)  \rightarrow A\underrel{\otimes}{G}\bold\Lambda_{n-2}(G) \big)=0,
$$
\begin{eqnarray*}
\Ha_{n-1}^{\lambda}(G,A) &=& \Ker \big(\partial_{n-1} : A\underrel{\otimes}{G}\bold\Lambda_{n-1}(G)
\rightarrow A\underrel{\otimes}{G}\bold\Lambda_{n-2}(G) \big)\\
&=& A \otimes \alpha.
\end{eqnarray*}
\end{enumerate}
\end{corollary}
\bigskip

\subsection{Finite cyclic group}
Let  $G= \left\langle \, t  \mid  t^n=1 \, \right\rangle$ be a cyclic group of order $n$. Then its exterior chain complex is
$$
0 \longrightarrow \bold\Lambda_{n-1}(G) \stackrel{\partial_{n-1}}{\longrightarrow}
   \bold\Lambda_{n-2}(G) \stackrel{\partial_{n-2}}{\longrightarrow} \cdots
           \stackrel{\partial_{2}}{\longrightarrow} \bold\Lambda_{1}(G)
           \stackrel{\partial_{1}}{\longrightarrow} \bold\Lambda_{0}(G)
           \stackrel{\varepsilon}{\longrightarrow} \mathbb{Z}\longrightarrow 0,
$$
where $\bold\Lambda_{0}(G)  = \mathbb{Z}[G]$ and
$$\bold\Lambda_{k}(G)  =  \textrm{mod}_{\mathbb{Z}[G]} \big\langle     \,  1\wedge t^{p_1}\wedge \cdots \wedge t^{p_k} \,| \, 1 \leq p_1 < \cdots < p_k \leq n-1    \big\rangle$$
for $1 \le k \le n-1$.

\begin{lemma} \label{l4}
The following hold:
\begin{enumerate}
\item
$ 1\wedge t\wedge \cdots \wedge \widehat{t^{p}}\wedge \cdots \wedge t^{n-1}= (-1)^{p(n-p-1)} ~t^{p+1} (1\wedge t\wedge \cdots \wedge t^{n-2})$.
\item $\bold\Lambda_{n-2}(G) =     \mod_{\mathbb{Z}[G]}\left\langle     \,  1\wedge t\wedge \cdots \wedge t^{n-2} \,     \right\rangle.$
\end{enumerate}
\end{lemma}

\begin{proof}
For assertion (1), we compute
\begin{eqnarray*}
  1\wedge t\wedge \cdots \wedge \widehat{t^{p}}\wedge \cdots \wedge t^{n-1} &=&    t^n\wedge t^{n+1}\wedge \cdots \wedge t^{n+p-1}\wedge t^{p+1} \wedge \cdots \wedge t^{n-1}\\
 &=&   (-1)^{p(n-p-1)}  (t^{p+1} \wedge \cdots \wedge t^{n-1} \wedge t^n\wedge t^{n+1}\wedge \cdots \wedge t^{n+p-1})\\
 &=&     (-1)^{p(n-p-1)}t^{p+1}( 1\wedge t\wedge \cdots \wedge t^{n-2}).
\end{eqnarray*}
Assertion (2) follows from (1).
\end{proof}

\medskip

Let $A$ be a right $G$-module. Since  $\bold\Lambda_{n}(G) =0$, we have
$$
\Ha_{n-1}^{\lambda}(G,A)=
\Ker \big(\partial_{n-1} : A\underrel{\otimes}{G}\bold\Lambda_{n-1}(G) \rightarrow A\underrel{\otimes}{G}\bold\Lambda_{n-2}(G) \big).
$$

\begin{lemma} \label{l5}
The following formula holds
$$
\partial_{n-1}(1\wedge t\wedge \cdots  \wedge t^{n-1})=
  \left( \sum\limits_{p=0}^{n-1}(-1)^{p(n-p)}t^{p+1} \right)( 1\wedge t\wedge \cdots  \wedge t^{n-2} ).
$$
\end{lemma}

\begin{proof}
We directly compute
\begin{eqnarray*}
\partial_{n-1}(1\wedge t\wedge \cdots  \wedge t^{n-1}) & = &
  \sum\limits_{p=0}^{n-1}(-1)^{p}    ( 1\wedge t\wedge \cdots \wedge \widehat{t^{p}}\wedge \cdots \wedge t^{n-1} )\\
& = & \left( \sum\limits_{p=0}^{n-1}(-1)^{p}(-1)^{p(n-p-1)} t^{p+1} \right)    ( 1\wedge t\wedge \cdots \wedge \cdots \wedge t^{n-2} )\\
& = & \left( \sum\limits_{p=0}^{n-1}(-1)^{p(n-p)}t^{p+1} \right)( 1\wedge t\wedge \cdots  \wedge t^{n-2} ).
\end{eqnarray*}
\end{proof}

Note that if  $n \equiv 1~(\mathrm{mod}~2)$, then $p(n-p) \equiv 0~(\mathrm{mod}~2)$ for
all $p=0,1,\ldots, n-1$. Similarly, if $n \equiv 0~(\mathrm{mod}~2)$, then $p(n-p) \equiv p^2 \equiv p~(\mathrm{mod}~2)$ for
all $p=0,1,\ldots, n-1$. Thus,
$$
\sum\limits_{p=0}^{n-1}(-1)^{p(n-p)}t^{p+1}=
\left\{
\begin{array}{lll}
  \sum\limits_{k=0}^{n-1} t^k           & \mbox{if} & n \equiv 1~(\mathrm{mod}~2 ), \\
  &&\\
  \sum\limits_{k=0}^{n-1} (-1)^{k+1}t^k & \mbox{if} & n \equiv 1~(\mathrm{mod}~ 2 ). \\
\end{array}
\right.
$$
Hence, we have
$$
\partial_{n-1}(1\wedge t\wedge \cdots  \wedge t^{n-1})=
\left\{
\begin{array}{lll}
  \left(\sum\limits_{k=0}^{n-1} t^k\right)( 1\wedge t\wedge \cdots  \wedge t^{n-2} )           & \mbox{if} & n \equiv 1~(\mathrm{mod}~2 ), \\
  & &\\
  \left(\sum\limits_{k=0}^{n-1} (-1)^{k+1}t^k\right)( 1\wedge t\wedge \cdots  \wedge t^{n-2} ) & \mbox{if} & n \equiv 1~(\mathrm{mod}~2 ). \\
\end{array}
\right.
$$
The preceding formula for the map $\partial_{n-1} : A\underrel{\otimes}{G}\bold\Lambda_{n-1}(G)  \longrightarrow A\underrel{\otimes}{G}\bold\Lambda_{n-2}(G)
$ gives

\begin{prop} \label{p2}
Let $A$ be a right $G$-module.
\begin{enumerate}
\item If $n \equiv 1~(\mathrm{mod}~2)$, then
$$
\im \big(\partial_{n-1} : A\underrel{\otimes}{G}\bold\Lambda_{n-1}(G)  \rightarrow A\underrel{\otimes}{G}\bold\Lambda_{n-2}(G) \big)=
\left\{
\, \left(a \sum\limits_{k=0}^{n-1} t^k\right)\otimes ( 1\wedge t\wedge \cdots  \wedge t^{n-2} )\, \mid a \in A\,
\right\},
$$
\begin{eqnarray*}
\Ha_{n-1}^{\lambda}(G,A) &=& \Ker \big(\partial_{n-1} : A\underrel{\otimes}{G}\bold\Lambda_{n-1}(G)
\rightarrow A\underrel{\otimes}{G}\bold\Lambda_{n-2}(G) \big)\\
&=&  \left\{\, a \otimes ( 1\wedge t\wedge \cdots  \wedge t^{n-1} )\, \mid a \in A , \,\,\left(a \sum\limits_{k=0}^{n-1} t^k\right)\otimes ( 1\wedge t\wedge \cdots  \wedge t^{n-2} )=0 \,\right\}.
\end{eqnarray*}
\item If $n \equiv 0~(\mathrm{mod}~2)$, then
$$
\im \big(\partial_{n-1} : A\underrel{\otimes}{G}\bold\Lambda_{n-1}(G)  \rightarrow A\underrel{\otimes}{G}\bold\Lambda_{n-2}(G) \big)=
\left\{
\, \left(a \sum\limits_{k=0}^{n-1} (-1)^{k+1}t^k\right)\otimes ( 1\wedge t\wedge \cdots  \wedge t^{n-2} )\, \mid a \in A\,
\right\},
$$
\begin{eqnarray*}
\Ha_{n-1}^{\lambda}(G,A)&=& \left\{ \, a \otimes ( 1\wedge t\wedge \cdots  \wedge t^{n-1} )\, \mid a \in A , \,\, \left(a \sum\limits_{k=0}^{n-1}(-1)^{k+1} t^k\right)\otimes ( 1\wedge t\wedge \cdots  \wedge t^{n-2} )=0 \,\right\}.
\end{eqnarray*}
\end{enumerate}
\end{prop}

\begin{prop} \label{p3}
Let  $A$ be a trivial $G$-module.
\begin{enumerate}
\item If $n \equiv 1~(\mathrm{mod}~2 )$, then
$$
\im \big(\partial_{n-1} : A\underrel{\otimes}{G}\bold\Lambda_{n-1}(G)  \rightarrow A\underrel{\otimes}{G}\bold\Lambda_{n-2}(G) \big)=
\left\{
\, n a \otimes ( 1\wedge t\wedge \cdots  \wedge t^{n-2} )\, |
\, a \in A\,
\right\},
$$
\begin{eqnarray*}
\Ha_{n-1}^{\lambda}(G,A) &=& \Ker \big(\partial_{n-1} : A\underrel{\otimes}{G}\bold\Lambda_{n-1}(G)
\rightarrow A\underrel{\otimes}{G}\bold\Lambda_{n-2}(G) \big)\\
 &=&
\left\{\, a \otimes ( 1\wedge t\wedge \cdots  \wedge t^{n-1} )\, \mid a \in A , \,\,na \otimes ( 1\wedge t\wedge \cdots  \wedge t^{n-2} )=0 \,\right\}.
\end{eqnarray*}
\item[]
\item If $n \equiv 0~(\mathrm{mod}~2 )$, then
$$
\im \big(\partial_{n-1} : A\underrel{\otimes}{G}\bold\Lambda_{n-1}(G)  \rightarrow A\underrel{\otimes}{G}\bold\Lambda_{n-2}(G) \big)=0,
$$
\begin{eqnarray*}
\Ha_{n-1}^{\lambda}(G,A)&=& \Ker \big(\partial_{n-1} : A\underrel{\otimes}{G}\bold\Lambda_{n-1}(G)
\rightarrow A\underrel{\otimes}{G}\bold\Lambda_{n-2}(G) \big)\\
&=& \big\{ a \otimes ( 1\wedge t\wedge \cdots  \wedge t^{n-1} )\, \mid a \in A , \,\,na \otimes ( 1\wedge t\wedge \cdots  \wedge t^{n-2} )=0 \big\}.
\end{eqnarray*}
\end{enumerate}
\end{prop}

Note that if  $n \equiv 0~(\mathrm{mod}~2)$, then
$$
a \otimes t( 1\wedge t\wedge \cdots  \wedge t^{n-1} )=
a \otimes (-1)^{n-1}( 1\wedge t\wedge \cdots  \wedge t^{n-1} )=
-a \otimes ( 1\wedge t\wedge \cdots  \wedge t^{n-1} ).
$$
In particular, for a trivial  $G$-module $A$, we have
$$
2a \otimes t( 1\wedge t\wedge \cdots  \wedge t^{n-1} )=0.
$$
Hence, if $A$ is a trivial $G$-module, then $\Ha_{n-1}^{\lambda}(G,A)$  is homomorphic image of the group  $A/2A$.
\par

\begin{conjecture}
Let $G$ be a cyclic group of order $n$ and $A$ a trivial  $G$-module.
\begin{enumerate}
\item If $n \equiv 0~(\mathrm{mod}~2)$, then
$$
\Ha_{n-1}^{\lambda}(G,A)\cong A/2A,
$$
\item If $n \equiv 1~(\mathrm{mod}~2)$, then
$$
\Ha_{n-1}^{\lambda}(G,A)\cong \Ker (\varphi_n : A \rightarrow A),
$$
where $\varphi_n (a)=na$ for $a\in A$.
\end{enumerate}
\end{conjecture}
\medskip

\subsection{Cyclic groups of order 3 and 4}

Next we compute exterior homology of cyclic groups of order 3 and 4.

\begin{prop}
If $G=\left\langle \, g \mid \, g^3=1 \,  \right\rangle$, then
\begin{eqnarray*}
\Ha_{0}^{\lambda}(G,A)&\cong & A/(A\, \Delta[G]) \cong A_G,\\
\Ha_{1}^{\lambda}(G,A)&\cong & A^G /A(1+g+g^2),\\
\Ha_{2}^{\lambda}(G,A) &= & \big\{  a\in A \mid  a(1+g+g^2)=0 \big\}.
\end{eqnarray*}
In particular, if $A$ is a trivial right $G$-module, then
\begin{eqnarray*}
\Ha_{0}^{\lambda}(G,A) &\cong & A,\\
\Ha_{1}^{\lambda}(G,A)&\cong & A /3A,\\
\Ha_{2}^{\lambda}(G,A) &= &\mathrm{Tor}_3(A)= \big\{  a\in A \mid  3a=0 \big\}.
\end{eqnarray*}
\end{prop}

\begin{proof}
The exterior chain complex for $G$ has the form
$$
0 \longrightarrow\bold\Lambda_{2}(G) \stackrel{\partial_2}{\longrightarrow}
   \bold\Lambda_{1}(G) \stackrel{\partial_1}{\longrightarrow} \bold\Lambda_{0}(G)  \stackrel{\varepsilon}{\longrightarrow} \mathbb{Z}\longrightarrow 0,
$$
where
$$
\bold\Lambda_{0}(G) =\mathbb{Z}[G], \quad
\bold\Lambda_{1}(G) =\textrm{mod}_{\mathbb{Z}[G]}\left\langle 1\wedge g, 1\wedge g^2  \right\rangle,  \quad
\bold\Lambda_{2}(G) =\textrm{mod}_{\mathbb{Z}[G]}\left\langle 1\wedge g\wedge g^2  \right\rangle.
$$
Notice that  $g (1\wedge g^2)= g\wedge g^3=g\wedge 1=- 1\wedge g$. Thus, we have
$$
\bold\Lambda_{1}(G) =\textrm{mod}_{\mathbb{Z}[G]}\left\langle 1\wedge g  \right\rangle,
$$
and hence $\bold\Lambda_{i}(G) $ are cyclic modules. Moreover, $\bold\Lambda_{0}(G)$ and $\bold\Lambda_{1}(G) $ are free $G$-modules whereas the module
$\bold\Lambda_{2}(G) $ is not free since it has the relation
$$
g (1\wedge g\wedge g^2) =g\wedge g^2\wedge 1 =1\wedge g\wedge g^2.
$$
Thus,  $\bold\Lambda_{2}(G) \cong\mathbb{Z}$ is the trivial $G$-module.
Further
$$
\partial_1(1\wedge g)=g-1,\quad \partial_2(1\wedge g\wedge g^2)=(1+g+g^2)(1\wedge g).
$$

For a right $G$-module $A$, we determine the homology of the chain complex
$$
0 \longrightarrow A\underrel{\otimes}{G}\bold\Lambda_{2}(G)
     \longrightarrow A\underrel{\otimes}{G}\bold\Lambda_{1}(G)
       \longrightarrow  A\underrel{\otimes}{G}\bold\Lambda_{0}(G)  \longrightarrow  0.
$$
An easy computation gives
$$
A\underrel{\otimes}{G}\bold\Lambda_{0}(G)  \cong A, \quad
A\underrel{\otimes}{G}\bold\Lambda_{1}(G)  \cong A, \quad
A\underrel{\otimes}{G}\bold\Lambda_{2}(G)  \cong A_G, \quad
$$
and
\begin{eqnarray*}
\im \partial_1 &=& \big\{ a(g-1) \mid a\in A  \big\},\\
\Ker \partial_1 &=&  \big\{ a\otimes (1\wedge g) \mid a(g-1)=0,\,\,a\in A  \big\},\\
\im \partial_2  &=& \big\{ a(1+g+g^2)\otimes (1\wedge g) \mid a\in A  \big\},\\
\Ker \partial_2  &=& \big\{ a\otimes (1\wedge g\wedge g^2) \mid a\in A,\,\, a(1+g+g^2)=0  \big\}.
\end{eqnarray*}
Thus, we obtain
\begin{eqnarray*}
\Ha_{0}^{\lambda}(G,A)&\cong & A/(A\,\Delta[G]) \cong A_G,\\
\Ha_{1}^{\lambda}(G,A)&\cong & A^G /A(1+g+g^2),\\
\Ha_{2}^{\lambda}(G,A) &= & \big\{  a\in A \mid  a(1+g+g^2)=0 \big\}.
\end{eqnarray*}
In particular, for a trivial right $G$-module $A$, we get
\begin{eqnarray*}
\Ha_{0}^{\lambda}(G,A) &\cong & A,\\
\Ha_{1}^{\lambda}(G,A)&\cong & A /3A,\\
\Ha_{2}^{\lambda}(G,A) &= &\mathrm{Tor}_3(A)= \big\{  a\in A \mid  3a=0 \big\}.
\end{eqnarray*}
\end{proof}

\begin{prop}
If $G=\left\langle \, g \, \mid \, g^4=1 \,  \right\rangle$, then
\begin{eqnarray*}
\Ha_{0}^{\lambda}(G,A)&\cong & A/(A\,\Delta[G]) \cong A_G,\\
\Ha_{1}^{\lambda}(G,A) &=& A \otimes (1\wedge g),\\
\Ha_{2}^{\lambda}(G,A) &=&  \left\{  a\otimes (1\wedge g\wedge g^2) \mid a\in A,\,\,  a\otimes (2(1\wedge g)-(1\wedge g^2))=0  \right\},\\
\Ha_{3}^{\lambda}(G,A) &=&   \left\{  a\otimes (1\wedge g\wedge g^2\wedge g^3) \mid a\in A  \right\},
\end{eqnarray*}
and there exists an epimorphism $ A/2A \longrightarrow \Ha_{3}^{\lambda}(G,A)$.
\par
In particular, if $A$ is a trivial $G$-module,  then
\begin{eqnarray*}
\Ha_{0}^{\lambda}(G,A) &\cong&  A/(A\,\Delta[G])\cong A,\\
\Ha_{1}^{\lambda}(G,A) &= & A \otimes (1\wedge g),\\
\Ha_{2}^{\lambda}(G,A) &=&  \left\{  a\otimes (1\wedge g\wedge g^2) \mid a\in A,\,\,  a\otimes (2(1\wedge g)-(1\wedge g^2))=0  \right\} ,\\
\Ha_{3}^{\lambda}(G,A) &=&   \left\{  a\otimes (1\wedge g\wedge g^2\wedge g^3) \mid a\in A  \right\}.
\end{eqnarray*}
Further, there exists an epimorphism  $A \longrightarrow \Ha_{3}^{\lambda}(G,A)$ given by
$a \mapsto a\otimes (1\wedge g\wedge g^2\wedge g^3)$ such that its kernel contains the submodule  $2A$.
\end{prop}

\begin{proof}
The exterior chain complex for $G$ has the form
$$
0 \longrightarrow \bold\Lambda_{3}(G) \longrightarrow \bold\Lambda_{2}(G) \longrightarrow \bold\Lambda_{1}(G) \longrightarrow
  \bold\Lambda_{0}(G) \longrightarrow \mathbb{Z}\longrightarrow 0,
$$
where
\begin{eqnarray*}
  \bold\Lambda_{0}(G)  & =&\mathbb{Z}[G],\\
    \bold\Lambda_{1}(G) & =& \mathrm{mod}_{\mathbb{Z}[G]}\left\langle     1\wedge g, 1\wedge g^2,  1\wedge g^3, g\wedge g^2,g\wedge g^3, g^2\wedge g^3     \right\rangle,\\
  \bold\Lambda_{2}(G)  & =&  \mathrm{mod}_{\mathbb{Z}[G]}\left\langle  g\wedge g^2\wedge g^3, 1\wedge g^2\wedge g^3, 1\wedge g\wedge g^3,1\wedge g\wedge g^2  \right\rangle,\\
  \bold\Lambda_{3}(G)  & =& \mathrm{mod}_{\mathbb{Z}[G]}\left\langle  1\wedge g\wedge g^2\wedge g^3  \right\rangle.
\end{eqnarray*}
Since we have the identities
\begin{eqnarray*}
1\wedge g^3 & =&  g^4\wedge g^3=-g^3 (1\wedge g),\\
g\wedge g^2 & =&  g (1\wedge g),\\
g\wedge g^3 & =&  g (1\wedge g^2),\\
g^2\wedge g^3 & =&  g^2 (1\wedge g),\\
g\wedge g^2\wedge g^3 & =&  g (1\wedge g\wedge g^2),\\
1\wedge g^2\wedge g^3 & =&  g^4\wedge g^2\wedge g^3=g^2 (1\wedge g\wedge g^2),\\
1\wedge g\wedge g^3 & =&  g^4\wedge g^5\wedge g^3=g^3 (1\wedge g\wedge g^2),
\end{eqnarray*}
it follows that
\begin{eqnarray*}
  \bold\Lambda_{1}(G) & =&  \mathrm{mod}_{\mathbb{Z}[G]}\left\langle     1\wedge g, 1\wedge g^2     \right\rangle,\\
  \bold\Lambda_{2}(G) & =&  \mathrm{mod}_{\mathbb{Z}[G]}\left\langle  1\wedge g\wedge g^2  \right\rangle.
\end{eqnarray*}
Note that the module  $  \bold\Lambda_{1}(G) $ is not free since it has the relation  $g^2 (1\wedge g^2)=g^2\wedge 1=-(1\wedge g^2)$. Next we compute the boundary maps
\begin{eqnarray*}
\partial_1(1\wedge g) &= & g-1,\\
 \partial_1(1\wedge  g^2) &= & g^2-1,\\
\partial_2(1\wedge g\wedge g^2) &= & (1+g)(1\wedge g)-1\wedge  g^2,\\
\partial_3(1\wedge g\wedge g^2\wedge g^3) &= & (-1+g-g^2+g^3)(1\wedge g\wedge g^2).
\end{eqnarray*}

For a right $G$-module  $A$, we now compute the homology of the chain complex
$$
0 \longrightarrow A\underrel{\otimes}{G}  \bold\Lambda_{3}(G)
     \longrightarrow A\underrel{\otimes}{G}  \bold\Lambda_{2}(G)
       \longrightarrow A\underrel{\otimes}{G}  \bold\Lambda_{1}(G)
          \longrightarrow   A\underrel{\otimes}{G}  \bold\Lambda_{0}(G)
              \longrightarrow  0.
$$
A direct check shows that $A\underrel{\otimes}{G}  \bold\Lambda_{0}(G)  \cong  A$ and
\begin{eqnarray*}
\im \partial_1 &=& \left\{ a(g-1) \mid a\in A  \right\},\\
\Ker \partial_1 &=& \left\{    a\otimes (1\wedge g)+b\otimes (1\wedge g^2) \mid a(g-1)+b(g^2-1)=0,\,\,a,b \in A    \right\},\\
\im  \partial_2 &=& \left\{  a\otimes ((1+g)(1\wedge g)-(1\wedge g^2)) \mid a\in A  \right\},\\
\Ker \partial_2 &=& \left\{  a\otimes (1\wedge g\wedge g^2) \mid a\in A,\,\,  a\otimes ((1+g)(1\wedge g)-(1\wedge g^2))=0  \right\},\\
\im  \partial_3 &=& \left\{  a(-1+g-g^2+g^3) \otimes (1\wedge g\wedge g^2) \mid a\in A  \right\},\\
\Ker \partial_3 &=& \left\{  a\otimes (1\wedge g\wedge g^2\wedge g^3) \mid a\in A,\,\,  a(-1+g-g^2+g^3) \otimes (1\wedge g\wedge g^2)=0  \right\}.
\end{eqnarray*}
This gives
\begin{eqnarray*}
\Ha_{0}^{\lambda}(G,A) &\cong&  A/(A\,\Delta[G])\cong A_G,\\
\Ha_{1}^{\lambda}(G,A) &=& \frac{  \left\{  a\otimes (1\wedge g)+b\otimes (1\wedge g^2) \mid a(g-1)+b(g^2-1)=0,\,\,a,b \in A\right\}}
{ \left\{  a\otimes ((1+g)(1\wedge g)-(1\wedge g^2)) \mid a\in A  \right\}},\\
\Ha_{2}^{\lambda}(G,A) &=& \frac{\left\{ a\otimes (1\wedge g\wedge g^2) \mid a\in A,\,\,  a\otimes ((1+g)(1\wedge g)-(1\wedge g^2))=0 \right\} }
{   \left\{  a(g^2+1)(g-1) \otimes (1\wedge g\wedge g^2) \mid a\in A \right\}},\\
\Ha_{3}^{\lambda}(G,A) &= &
   \left\{
  a\otimes (1\wedge g\wedge g^2\wedge g^3) \mid a\in A,\,\,  a(g^2+1)(g-1) \otimes (1\wedge g\wedge g^2)=0
  \right\}.
\end{eqnarray*}
\par
Finally, suppose that $A$ is a trivial  $G$-module. Since
$$
a\otimes (1\wedge g)+b\otimes (1\wedge g^2)+b\otimes (2(1\wedge g)-(1\wedge g^2))=
 (a+2b)\otimes (1\wedge g)
$$
and  $a+2b$ is an arbitrary element of  $A$, we obtain
\begin{eqnarray*}
\Ha_{1}^{\lambda}(G,A) &=& A \otimes (1\wedge g),\\
\Ha_{2}^{\lambda}(G,A) &=&  \left\{  a\otimes (1\wedge g\wedge g^2) \mid a\in A,\,\,  a\otimes (2(1\wedge g)-(1\wedge g^2))=0  \right\},\\
\Ha_{3}^{\lambda}(G,A) &=&   \left\{  a\otimes (1\wedge g\wedge g^2\wedge g^3) \mid a\in A  \right\}.
\end{eqnarray*}
Further, since
$$
g(1\wedge g\wedge g^2\wedge g^3) =-(1\wedge g\wedge g^2\wedge g^3),
$$
we have
$$
2a\otimes (1\wedge g\wedge g^2\wedge g^3) =0.
$$
Hence, there exists an epimorphism
$ A/2A \longrightarrow \Ha_{3}^{\lambda}(G,A).$
\end{proof}
\medskip

We conclude this section with some results on exterior cohomology. Let $G$ be a finite group of order $n$ and $A$ a left $G$-module.
Applying $\Hom_G(-, A)$ functor on the exterior chain complex
$$
0 \longrightarrow   \bold\Lambda_{n-1}(G)
   \stackrel{\partial_{n-1}}{\longrightarrow} \bold\Lambda_{n-2}(G)
     \stackrel{\partial_{n-2}}{\longrightarrow}   \cdots
     \stackrel{\partial_{2}}{\longrightarrow} \bold\Lambda_{1}(G)
       \stackrel{\partial_{1}}{\longrightarrow}  \bold\Lambda_{0}(G)
        \stackrel{\varepsilon}{\longrightarrow} \mathbb{Z}\longrightarrow 0,
$$
gives the cochain complex
$$
0\longrightarrow  \Hom_G\big(\bold\Lambda_{0}(G) , A\big)  \stackrel{\delta^{0}}{\longrightarrow}  \Hom_G\big(\bold\Lambda_{1}(G) , A\big)
    \stackrel{\delta^{1}}{\longrightarrow}
  \cdots
$$
$$
 \cdots \stackrel{\delta^{n-3}}{\longrightarrow}  \Hom_G\big(\bold\Lambda_{n-2}(G) , A\big)
      \stackrel{\delta^{n-2}}{\longrightarrow} \Hom_G\big(\bold\Lambda_{n-1}(G) , A\big)
   \longrightarrow 0,
$$
where the coboundary map   $\delta^k$ is induced by the boundary map  $\partial_{k+1}$. This gives

$$
\Ha^{n-1}_{\lambda}(G,A)= \Hom_G\big(\bold\Lambda_{n-1}(G) , A\big)/ \im ( \delta^{n-2}).
$$
If  $f\in \Hom_G\big(\bold\Lambda_{n-2}(G) , A\big)$, then
$$
\delta^{n-2} f (\omega)=f\big(\partial_{n-1}(\omega)\big)
$$
where $\omega \in \bold\Lambda_{n-1}(G)$. Since
$\bold\Lambda_{n-1}(G) = \mathrm{mod}_{\mathbb{Z}[G]}\left\langle \alpha \right\rangle$, using the formula for
 $\partial_{n-1}(\alpha)$, we get
$$
\delta^{n-2} f (g \alpha)=
\left\{
\begin{array}{ll}
  g N f(\beta)  & \mbox{if}\,\, G \,\, \mbox{is oriented}, \\
  &\\
  g \left(\sum\limits_{i=1}^{n}\mathrm{sign}\big(\kappa(g_i)\big)g_i\right)f(\beta) &
   \mbox{if}\,\, G \,\, \mbox{is non-oriented}. \\
\end{array}
\right.
$$
If $N=\sum_{g \in G} g$ is the norm element, then $g N = N$ for each $g \in G$. If $G$ is non-oriented, then 
$$
g\left(\sum\limits_{i=1}^{n}\mathrm{sign}\big(\kappa(g_i)\big)g_i\right)=
\mathrm{sign}\big(\kappa(g)\big)\left(\sum\limits_{i=1}^{n}\mathrm{sign}\big(\kappa(g_i)\big)g_i\right)
$$
for all $g \in G$, and hence
$$
\delta^{n-2} f (g \alpha)=
\left\{
\begin{array}{ll}
 N f(\beta)  & \mbox{if}\,\, G \,\, \mbox{is oriented}, \\
  \mathrm{sign}\big(\kappa(g)\big) \left(\sum\limits_{i=1}^{n}\mathrm{sign}\big(\kappa(g_i)\big)g_i\right)f(\beta) &
   \mbox{if}\,\, G \,\, \mbox{is non-oriented}. \\
\end{array}
\right.
$$
If $A$ is a trivial left $G$-module, then
$$
\delta^{n-2} f (g \alpha)=
\left\{
\begin{array}{ll}
  n f(\beta)  & \mbox{if}\,\, G \,\, \mbox{is oriented }, \\
   0 &   \mbox{if}\,\, G \,\, \mbox{is non-oriented}.
\end{array}
\right.
$$

Thus, we obtain the following result.

\begin{theorem} \label{t4}
If $f\in \Hom_G\big(\bold\Lambda_{n-2}(G) , A\big)$, then
$$
\delta^{n-2} f (g \alpha)=
\left\{
\begin{array}{ll}
  N f(\beta)  & \mbox{if}\,\, G \,\, \mbox{is oriented}, \\
  \mathrm{sign}\big(\kappa(g)\big) \left(\sum\limits_{i=1}^{n}\mathrm{sign}\big(\kappa(g_i)\big)g_i\right)f(\beta) &
   \mbox{if}\,\, G \,\, \mbox{is non-oriented}. \\
\end{array}
\right.
$$

In particular, if  $A$ is a trivial right  $G$-module, then
$$
\delta^{n-2} f (g \alpha)=
\left\{
\begin{array}{ll}
  n f(\beta)  & \mbox{if}\,\, G \,\, \mbox{is oriented}, \\
   0 &   \mbox{if}\,\, G \,\, \mbox{is non-oriented}. \\
\end{array}
\right.
$$
\end{theorem}

As a consequence of the preceding result, we have

\begin{corollary}
If $G$ is a non-oriented group and $A$ a trivial $G$-module, then
$$
\Ha^{n-1}_{\lambda}(G,A)=
   \Hom_G\big(\bold\Lambda_{n-1}(G) , A\big).
$$
\end{corollary}
\medskip

\section{(Co)restriction homomorphisms in cohomology}\label{sec8}
In this final section, we investigate restriction and corestriction homomorphisms for symmetric and exterior cohomologies of groups. Throughout the section, $H$ is a subgroup of a group $G$ and $A$ is a right $G$-module. In what follows, the cochain complex $\{\Ca^*(G,A), \delta^* \}$ is as in Subsection \ref{another-complex}.

\subsection{(Co)restriction  homomorphism in classical cohomology}\label{sec-cores-classical-cohom}
Since $A$ is a $G$-module, it can be viewed as an $H$-module. A projective resolution $\Ca_* \to \mathbb{Z}$ of the trivial $G$-module $\mathbb{Z}$ can be viewed as a projective resolution of the trivial $H$-module  $\mathbb{Z}$.
Hence the  natural homomorphism of cochain complexes
$$
\Hom_G(\Ca_*, A) \longrightarrow \Hom_H(\Ca_*, A)
$$
gives a homomorphism of cohomology groups
$$
\res^G_H: \Ha^n(G, A) \longrightarrow \Ha^n(H, A),
$$
for each $n \ge 0$, called the {\it restriction  homomorphism}.
\bigskip

Suppose that $H$ is a subgroup of $G$ of finite index $k$. Let $\{c_1, c_2, \ldots, c_k \}$ be a fixed set of representatives of left cosets of $H$ in $G$. Then $G= \bigcup_{i=1}^k c_iH$. By convention if $cH=H$, then $c=1$. For an element $g \in G$, let $\bar{g}$ denote the unique coset representative $c_i$ such that $c_i H=gH$. If $g_1,\ldots,g_n \in G$, we set the notations
$$x_1=g_1\ldots g_n,~x_2=g_2\ldots g_n,~\ldots,~x_n=g_n.$$
It is well-known \cite[Proposition 2.5.1]{Weiss} that there is a  natural homomorphism of cochain complexes $$\tr^*: \Ca^*(H,A)\longrightarrow \Ca^*(G,A),$$ which for each $n \ge 0$ is  given by
\begin{equation}\label{sigma-prime}
\tr^n(\sigma)(g_1,\ldots,g_n)=\sum_{i=1}^k\overline{x_1c_i}~\sigma\big(\overline{x_1c_i}^{-1}g_1\overline{x_2c_i},\overline{x_2c_i}^{-1}g_2\overline{x_3c_i},\ldots,\overline{x_nc_i}^{-1}g_n\overline{c_i}\big)
\end{equation}
for $g_1,\ldots,g_n\in G$ and $\sigma\in \Ca^n(H, A)$. This yields the {\it corestriction homomorphism} $$\cores_H^G:\Ha^*(H,A)\longrightarrow \Ha^*(G,A)$$ given by $$\cores_H^G\big([\sigma] \big)=\big[\tr^n(\sigma)\big],$$
where  $\sigma\in \Za^n(H, A)$, the group of $n$-cocycles. Notice that $\overline{x_ic_j}^{-1}g_i\overline{x_{i+1}c_j}\in H$ and $\overline{x_nc_j}^{-1}g_n\overline{c_j}\in H$ for each $i\in\{1,\ldots,n-1\}$ and $j \in \{1,\ldots,k\}$.

\begin{remark}
The restriction and the corestriction homomorphisms for the classical cohomology of groups can also be defined using the Eckman-Shapiro \cite[Proposition 6.2]{Brown}, which crucially depends on the fact that the resolutions are free. However, this approach does not work for our purpose since the resolutions used for defining exterior and symmetric cohomology need not be free in general.
\end{remark}

\begin{remark}\label{Trace-rem}
We can interpret the preceding  explicit construction of the corestriction homomorphism for the cochain complex \eqref{eqno3}. This will be useful in defining corestriction homomorphism for symmetric and exterior cohomology.  Recall the isomorphism \eqref{iso-two-complexes}
\begin{equation*}
\psi^n:\Hom_G\big(\mathbb{Z}[G^{n+1}], A\big) \longrightarrow \Ca^n(G,A).
\end{equation*}
For each $n \ge 0$, define $$\Tr^n:\Hom_H\big(\mathbb{Z}[H^{n+1}], A\big)\longrightarrow \Hom_G\big(\mathbb{Z}[G^{n+1}], A\big)$$ as  $$\Tr^n=(\psi^n)^{-1} \circ \tr^n \circ \psi^n.$$
More precisely, for $g_0, g_1, \ldots, g_n \in G$ and $\sigma \in \Hom_H\big(\mathbb{Z}[H^{n+1}], A\big)$, we have

\begin{equation}\label{Trace-formula}
\end{equation}
\begin{small}
\begin{eqnarray*}
& & \Tr^n (\sigma)(g_0, g_1,\ldots, g_n)\\
&=& (\psi^n)^{-1} \circ \tr^n \circ \psi^n (\sigma)(g_0, g_1,\ldots, g_n)\\
&=& g_0.\tr^n \circ \psi^n (\sigma) (g_0^{-1}g_1, g_1^{-1}g_2, \ldots, g_{n-1}^{-1}g_n)\\
&=& g_0. \sum_{i=1}^k \overline{g_0^{-1} g_n c_i} ~\psi^n (\sigma) \Big(\overline{g_0^{-1} g_nc_i}^{-1} (g_0^{-1}g_1) \overline{g_1^{-1} g_nc_i}, ~\overline{g_1^{-1} g_nc_i}^{-1} (g_1^{-1}g_2) \overline{g_2^{-1} g_nc_i},~ \ldots,  ~\overline{g_{n-1}^{-1} g_n c_i}^{-1}  (g_{n-1}^{-1}g_n) \overline{c_i}  \Big)\\
&=&  g_0. \sum_{i=1}^k \overline{g_0^{-1} g_n c_i} ~\sigma \Big(1,~ \overline{g_0^{-1} g_nc_i}^{-1} (g_0^{-1}g_1) \overline{g_1^{-1} g_nc_i},~ \overline{g_0^{-1} g_nc_i}^{-1} (g_0^{-1}g_2) \overline{g_2^{-1} g_nc_i},~\ldots,\overline{g_0^{-1} g_nc_i}^{-1} (g_0^{-1}g_n) \overline{c_i} \Big),
\end{eqnarray*}
\end{small}
where $\overline{g_0^{-1} g_nc_i}^{-1} (g_0^{-1}g_t) \overline{g_2^{-1} g_nc_i}, ~\overline{g_0^{-1} g_nc_i}^{-1} (g_0^{-1}g_n) \overline{c_i}  \in H$ for each $1 \le t \le n$ and $1 \le i \le k$.
\end{remark}
\vspace*{1mm}

\subsection{(Co)restriction  homomorphism in symmetric cohomology}
Recall that, by Lemma \ref{staic-alternate}, the cohomology of the cochain complex $\{\KS^* (G,A), \delta^*\}$ is the symmetric cohomology $\Has^*(G,A)$. The  natural homomorphism of cochain complexes
$$
\KS^* (G,A)\longrightarrow \KS^* (H,A)
$$
gives the \textit{restriction homomorphism} of symmetric cohomology groups
$$
\sres^G_H: \Has^*(G, A) \longrightarrow\Has^*(H, A).
$$
See also \cite[Corollary 5.2]{Singh} for an alternate description. The direct construction of corestriction homomorphism for classical cohomology in Subsection \ref{sec-cores-classical-cohom} was used by Todea \cite{Todea} to define a corestriction homomorphism for symmetric cohomology.

\begin{prop}\label{cores-map-symmetric}
Let $H$ be a finite index subgroup of a group $G$ and $A$ a $G$-module. Then there is a corestriction homomorphism $$\scores_H^G:\Has^n(H,A)\longrightarrow \Has^n(G,A).$$
\end{prop}

\begin{proof}
For $n\geq 0$ and $\sigma\in \KS^n(H,A)$, it follows that $\Tr^n(\sigma)\in \KS^n(G,A)$. Further, as in \cite[Lemma 3.1]{Todea}, the following diagram commutes

\begin{displaymath}
 \xymatrix{\KS^n(H,A)\ar[rr]^{\sigma \mapsto \Tr^n(\sigma)}\ar[d]^{\delta^n} && \KS^n(G,A) \ar[d]^{\delta^n} \\
 \KS^{n+1}(H,A)\ar[rr]^{\sigma \mapsto \Tr^{n+1}(\sigma)}  &&\KS^{n+1}(G, A).   }
\end{displaymath}
We define  $$\scores_H^G:\Has^n(H,A)\longrightarrow \Has^n(G,A)$$ by setting
$$\scores_H^G\big([\sigma]\big)=\big[\Tr^n(\sigma)\big],$$
where $\sigma\in \KS^n(H,A)$ is a symmetric $n$-cocycle. Thus, $\scores_H^G$ is the desired corestriction homomorphism.
\end{proof}

\subsection{(Co)restriction homomorphism in exterior cohomology}
Recall that the cohomology of the cochain complex $\{\Ka^*_\lambda (G,A), \delta^* \}$ is the exterior cohomology $\Ha_\lambda^*(G,A)$. The  natural homomorphism of cochain complexes
$$
\Ka^*_\lambda(G,A)\longrightarrow \Ka^*_\lambda (H,A)
$$
gives the \textit{restriction homomorphism} of exterior cohomology groups
$$
\lres^G_H:\Ha_\lambda^*(G, A) \longrightarrow\Ha_\lambda^*(H, A).
$$
\par

\begin{prop}
Let $H$ be a finite index subgroup of a group $G$ and $A$ a $G$-module. Then there is a corestriction homomorphism $$\lcores_H^G:\Ha^n_\lambda(H,A)\longrightarrow \Ha^n_\lambda(G,A).$$
\end{prop}

\begin{proof}
Let $n\geq 0$ and $\sigma\in \Ka^n_\lambda(H,A)$. Then $\sigma(h_0,\ldots,h_i,h_i,\ldots,h_n)=0$ for all $0\leq i< n$  and $h_0, h_1, \ldots, h_n \in H$. It follows from the last equality in \eqref{Trace-formula} that if  $g_0, g_1, \ldots,g_n \in G$ with $g_j=g_{j+1}$  for some $0\leq j< n$, then $\Tr^n (\sigma)(g_0,\ldots,g_j,g_j,\ldots,g_n)=0$, and hence $\Tr^n(\sigma)\in \Ka^n_\lambda(G,A)$. In addition, as in Proposition \ref{cores-map-symmetric}, the following diagram commutes

\begin{displaymath}
 \xymatrix{\Ka^n_\lambda(H,A)\ar[rr]^{\sigma \mapsto \Tr^n(\sigma)}\ar[d]^{\delta^n} && \Ka^n_\lambda(G,A) \ar[d]^{\delta^n} \\
\Ka^{n+1}_\lambda(H,A)\ar[rr]^{\sigma \mapsto \Tr^{n+1}(\sigma)}  &&\Ka^{n+1}_\lambda(G,A).   }
\end{displaymath}
Thus, we can define the \textit{corestriction homomorphism} 
 $$\lcores_H^G:\Ha^n_\lambda(H,A)\longrightarrow \Ha^n_\lambda(G,A)$$ by setting
$$\lcores_H^G\big([\sigma]\big)=\big[\Tr^n(\sigma)\big],$$
where $\sigma\in \Ka^n_\lambda(H,A)$ is an exterior $n$-cocycle.
\end{proof}
\medskip

\begin{question}
We conclude with the following questions:
\begin{enumerate}
\item How are the groups $\Ha_{2}(G,\mathbb{Z})$,  $\Ha_{2}^{\lambda}(G,\mathbb{Z})$ and $\Has_{2}(G,\mathbb{Z})$ related, where  $\mathbb{Z}$ is a trivial $G$-module? In particular, is there a Hopf type formula for the second exterior and symmetric homologies?
\item  Do there exist restriction-corestriction formulas for exterior and symmetric (co)homologies?
\item  What can we say about the homomorphism $\lambda_* \iota_*$?
\end{enumerate}
\end{question}

\begin{ack}
The authors thank the anonymous referee for many useful comments which considerably improved the paper. Mariam Pirashvili is also thanked for her interest in this work. Bardakov and Neshchadim are supported by the Russian Science Foundation project no. 19-41-02005. Singh is supported by the SERB MATRICS Grant MTR/2017/000018.
\end{ack}

\end{document}